\documentclass{article}
\usepackage[letterpaper]{geometry}

\usepackage[T1]{fontenc}
\usepackage{amsmath}
\usepackage{amssymb}
\usepackage{amsthm}
\usepackage{url}
\usepackage{tikz}
\usepackage{subfig}
\usepackage{float}
\usepackage{enumerate}
\usepackage[pdfborder={0 0 0}]{hyperref}
\usepackage{scrextend}
\usepackage{color}

\newtheorem{thm}{Theorem}[section]
\newtheorem{defn}[thm]{Definition}

\newtheorem{cor}[thm]{Corollary}
\newtheorem{obs}[thm]{Observation}
\newtheorem{lem}[thm]{Lemma}
\newtheorem{prop}{Property}

\allowdisplaybreaks[1]

\long\def\symbolfootnote[#1]#2{\begingroup%
\def\thefootnote{\fnsymbol{footnote}}\footnote[#1]{#2}\endgroup}

\begin{document}

\title{An infinite class of unsaturated rooted trees corresponding to designable RNA secondary structures}
\author{Jonathan Jedwab \and Tara Petrie \and Samuel Simon}
\date{23 September 2017 (revised 8 May 2020)}
\maketitle

\symbolfootnote[0]{
Department of Mathematics, 
Simon Fraser University, 8888 University Drive, Burnaby BC V5A 1S6, Canada.
\par
J.~Jedwab and T. Petrie are supported by NSERC.
\par
Email: {\tt jed@sfu.ca}, {\tt tpetrie@sfu.ca}, {\tt ssimon@sfu.ca}
\par
The results of this paper form part of the Master's thesis of T. Petrie \cite{petrie-masters}, who presented them in part at the CanaDAM 2017 conference in Toronto, ON.
}

\begin{abstract}
An RNA secondary structure is designable if there is an RNA sequence which can attain its maximum number of base pairs only by adopting that structure.
The combinatorial RNA design problem, introduced by Hale{\v{s}} et al.\ in 2016, is to determine whether or not a given RNA secondary structure is designable.
Hale{\v{s}} et al.\ identified certain classes of designable and non-designable secondary structures by reference to their corresponding rooted trees.
We introduce an infinite class of rooted trees containing unpaired nucleotides at the greatest depth, and prove constructively that their corresponding secondary structures are designable. This complements previous results for the combinatorial RNA design problem.
\end{abstract}


\section{Introduction}
\label{sec:intro}

Ribonucleic acid (RNA) is a biomolecule which performs many roles in cellular organisms, including conveying genetic information, controlling protein synthesis, and catalyzing biological reactions. A strand of RNA comprises a chain of nucleotides, each of which is one of the nitrogenous bases guanine (G), cytosine (C), adenine~(A), and uracil~(U).
The three-dimensional spatial configuration of an RNA strand affects its biological function~\cite{esmaili2014evolutionary}.  
A first approximation to this configuration is given by the \emph{secondary structure}, namely the two-dimensional folding of the RNA strand onto itself under only the interaction of certain pairs of nucleotides in the sequences to form \emph{base pairs}, as illustrated in Figure~\ref{secondarypic}.
The three allowable base pair types are $\left\{{\text{G,C}}\right\}$, $\left\{{\text{A,U}}\right\}$, and (less frequently) $\left\{{\text{G,U}}\right\}$.

\begin{figure}[h]
\begin{center}
\includegraphics[width=8cm]{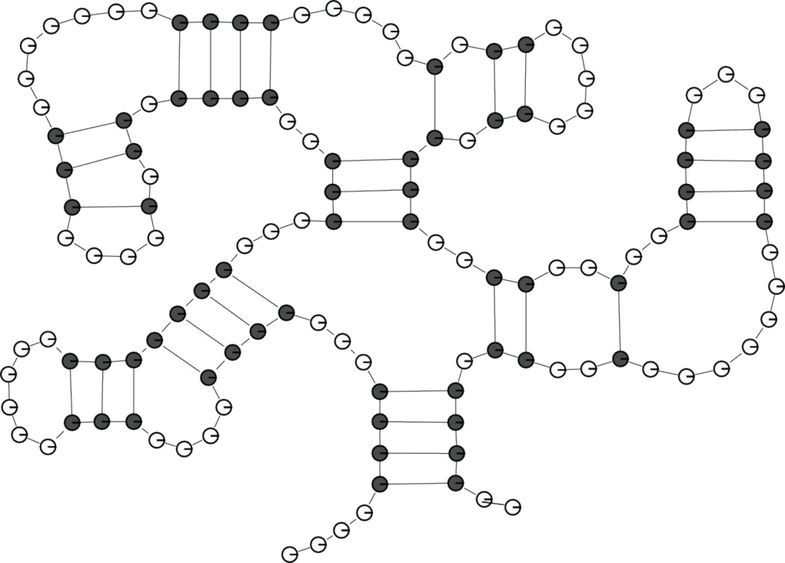}
\end{center}
\caption[RNA secondary structure]{An RNA secondary structure in which base pairs are represented by grey points. Reproduced from \cite{secondarypicture} under Creative Commons Attribution License.}
\label{secondarypic}
\end{figure}

Given a nucleotide sequence, the stability of the RNA secondary structure increases with the number of base pairs. 
Nussinov and Jacobson~\cite{nussinov1980fast} showed that, given an RNA sequence of $n$ nucleotides, a secondary structure with the maximum number of base pairs can be predicted in order $n^3$ time.
Modern algorithms for secondary structure prediction are presented in \cite{zuker1981optimal, mathews1999expanded}, for example, and are widely used in current software.

The inverse problem to secondary structure prediction is the \emph{RNA design problem}: find an RNA sequence which adopts a target secondary structure in preference to all other secondary structures, under some specified energy model.
There is an extensive literature on computational approaches to the RNA design problem as an optimization problem: see, for example, 
\cite{busch2006info, aguirre2007computational, taneda2011modena, zadeh2011nucleic, lyngso2012frnakenstein, garcia-martin2013, esmaili2014evolutionary, bonnet2017}.
Nonetheless, there is currently no known exact polynomial-time algorithm for solving the RNA design problem, and its complexity remains open~\cite{halevs2016combinatorial}; a more general problem is known to be NP-hard~\cite{schnall2008inverting}.

In view of these difficulties, Hale{\v{s}} et al. \cite{halevs2016combinatorial} proposed the \emph{combinatorial RNA design problem} as an idealized version of the RNA design problem: 
given a target secondary structure, find an RNA sequence which can achieve its maximum number of base pairs only by adopting the specified secondary structure, or else show that no such sequence exists.
This restricted problem is intended to be sufficiently tractable to allow algorithmic insights that could apply to more sophisticated models.
Hale{\v{s}} et al. \cite{halevs2016combinatorial} considered a number of energy models under which minimizing the total energy score for an RNA sequence corresponds to maximizing a weighted sum of base pair types.
In the \emph{Nussinov-Jacobson model}, the three base pair types \{G,C\}, \{A,U\}, \{G,U\} are assigned an energy score $\alpha$, $\beta$, $\gamma$, respectively, where $\max(\alpha,\beta) < \gamma < 0$; all other base pair types are assigned an energy score of $\infty$ in order to exclude them.
The \emph{Watson-Crick} model takes $\alpha = \beta = -1$ and $\gamma=\infty$, so that minimizing the total energy score is the same as maximizing the total number of \{G,C\} and \{A,U\} base pairs.
We shall use the Watson-Crick model throughout this paper (except that in Section~\ref{sec:results} we extend our main result to the Nussinov-Jacobson model with $\alpha = \beta = \gamma = -1$, which corresponds to maximizing the total number of \{G,C\} and \{A,U\} and \{G,U\} base pairs).

We point out that the immediate practical implications of our results are limited in two important respects.
Firstly, the class of ``P-unsaturated floral trees'' that we study (see Section~\ref{sec:results}) is highly restricted: it constrains multi-loops to have exactly three branches and to have no unpaired bases; it disallows internal loops and bulges; and it requires all unpaired bases in hairpin loops to occur at the same nesting depth with respect to the exterior loop.
Secondly, models for which minimization of the energy score is the same as maximization of the total number of base pairs have generally been found to have poor predictive ability, particularly for longer RNA sequences. This suggests it is unlikely that an RNA sequence identified by our results as a design for a target secondary structure would in practice adopt that secondary structure in preference to all other secondary structures. In Section~\ref{sec:future}, we propose as a topic for future study to attempt to extend the methods of this paper to provide predictions of greater practical relevance.

\section{Terminology and previous results}
\label{sec:terminology}

In this section we define some terminology and then summarize the main results of~\cite{halevs2016combinatorial}.
Consider the sequence of nucleotides
\begin{center}
\begin{tikzpicture}
\node at (1,1) {A};
\node at (1.5,1) {G};
\node at (2,1) {A};
\node at (2.5,1) {C};
\node at (3,1) {C};
\node at (3.5,1) {U};
\node at (4,1) {U};
\node at (4.5,1) {C};
\node at (5,1) {C};
\node at (5.5,1) {A};
\node at (6,1) {C};
\node at (6.5,1) {C};
\node at (7,1) {A};
\node at (7.5,1) {C};
\node at (8,1) {C};
\node at (8.5,1) {U};
\node at (9,1) {U};
\node at (9.5,1) {C};
\node at (10,1) {C};
\node at (10.5,1) {C};
\node at (11,1) {A};
\node at (11.5,1) {G};
\node at (12,1) {U.};
\end{tikzpicture}
\end{center}
\noindent 
This sequence admits a secondary structure involving 7 base pairs of types \{G,C\} and \{A,U\}, represented on the left side of Figure~\ref{mainproblem} by arcs that are all drawn above the nucleotide sequence.
\begin{figure}
\begin{center}
\begin{tikzpicture}[scale=0.7]
\node at (1,1) {A};
\node at (1.5,1) {G};
\node at (2,1) {A};
\node at (2.5,1) {C};
\node at (3,1) {C};
\node at (3.5,1) {U};
\node at (4,1) {U};
\node at (4.5,1) {C};
\node at (5,1) {C};
\node at (5.5,1) {A};
\node at (6,1) {C};
\node at (6.5,1) {C};
\node at (7,1) {A};
\node at (7.5,1) {C};
\node at (8,1) {C};
\node at (8.5,1) {U};
\node at (9,1) {U};
\node at (9.5,1) {C};
\node at (10,1) {C};
\node at (10.5,1) {C};
\node at (11,1) {A};
\node at (11.5,1) {G};
\node at (12,1) {U};
\draw[bend left=65] (1,1.25) to (12,1.25);
\draw[bend left=65] (1.5,1.25) to (6,1.25);
\draw[bend left=65] (6.5,1.25) to (11.5,1.25);
\draw[bend left=65] (2,1.25) to (3.5,1.25);
\draw[bend left=65] (4,1.25) to (5.5,1.25);
\draw[bend left=65] (7,1.25) to (8.5,1.25);
\draw[bend left=65] (9,1.25) to (11,1.25);
\end{tikzpicture}\hspace{0.5cm}
\begin{tikzpicture}[scale=0.4]
\node[scale=0.7, draw, circle](n21) at (1.5,5)  {C};
\node[scale=0.7, draw, circle](n22) at (5.5,5)  {C};
\node[scale=0.7, draw, circle](n23) at (9.5,5)  {C};
\node[scale=0.7, draw, circle](n24) at (13,5)  {C}; %
\node[scale=0.7, draw, circle](n25) at (3.5,5)  {C};
\node[scale=0.7, draw, circle](n26) at (7.5,5)  {C};
\node[scale=0.7, draw, circle](n27) at (11.5,5)  {C};
\node[scale=0.7, draw, circle](n28) at (16,5)  {C};
\node[scale=0.7, draw, circle](n29) at (14.5,5)  {C};
\node[scale=0.7, draw, circle](n31) at (2.5,7)  {AU};
\node[scale=0.7, draw, circle](n32) at (6.5,7)  {UA};
\node[scale=0.7, draw, circle](n33) at (10.5,7)  {AU};
\node[scale=0.7, draw, circle](n34) at (14.5,7)  {UA};
\node[scale=0.7, draw, circle](n41) at (4.5,9)  {GC};
\node[scale=0.7, draw, circle](n42) at (12.5,9)  {CG};
\node[scale=0.7, draw, circle](n51) at (8.5,11)  {AU};

\foreach \from/\to in {n51/n41, n51/n42, n41/n31, n41/n32, n42/n33, n42/n34, n31/n21, n32/n22, n33/n23, n34/n24,  n25/n31, n26/n32, n27/n33, n28/n34, n34/n29}
\draw (\from) -- (\to);
\end{tikzpicture}
\end{center}
\caption{Labelled secondary structure (left) and corresponding labelled tree representation (right)}
\label{mainproblem}
\end{figure}
The right side of Figure~\ref{mainproblem} shows the corresponding representation of this labelled secondary structure in labelled tree form. We map from the labelled tree to the labelled secondary structure by visiting tree vertices in the order given by 
recursively visiting the root, and then (if the root is an internal vertex) the vertices of the left subtree, the vertices of the right subtree, and the root again; and then joining by an arc each pair of nucleotides that label the same tree vertex. We map in the reverse direction by representing the outermost arc as the tree root AU, the two arcs nested directly beneath the outermost arc as a left child GC and a right child CG of the tree root AU, and so on.
In this way we associate each labelled secondary structure (a set of nucleotides and base pairs) with a labelled tree.

In Figure~\ref{mainproblemnolabels}, we have removed the nucleotide labels to leave the underlying secondary structure and the corresponding tree representation.
\begin{figure}
\begin{center}
\begin{tikzpicture}[scale=0.7]
\node at (1,1) {$\bullet$};
\node at (1.5,1) {$\bullet$};
\node at (2,1) {$\bullet$};
\node at (2.5,1) {$\bullet$};
\node at (3,1) {$\bullet$};
\node at (3.5,1) {$\bullet$};
\node at (4,1) {$\bullet$};
\node at (4.5,1) {$\bullet$};
\node at (5,1) {$\bullet$};
\node at (5.5,1) {$\bullet$};
\node at (6,1) {$\bullet$};
\node at (6.5,1) {$\bullet$};
\node at (7,1) {$\bullet$};
\node at (7.5,1) {$\bullet$};
\node at (8,1) {$\bullet$};
\node at (8.5,1) {$\bullet$};
\node at (9,1) {$\bullet$};
\node at (9.5,1) {$\bullet$};
\node at (10,1) {$\bullet$};
\node at (10.5,1) {$\bullet$};
\node at (11,1) {$\bullet$};
\node at (11.5,1) {$\bullet$};
\node at (12,1) {$\bullet$};
\draw[bend left=65] (1,1.25) to (12,1.25);
\draw[bend left=65] (1.5,1.25) to (6,1.25);
\draw[bend left=65] (6.5,1.25) to (11.5,1.25);
\draw[bend left=65] (2,1.25) to (3.5,1.25);
\draw[bend left=65] (4,1.25) to (5.5,1.25);
\draw[bend left=65] (7,1.25) to (8.5,1.25);
\draw[bend left=65] (9,1.25) to (11,1.25);
\end{tikzpicture}\hspace{0.5cm}
\begin{tikzpicture}[scale=0.4]
\node[scale=0.7, draw, circle](n21) at (1.5,5)  {$\bullet$};
\node[scale=0.7, draw, circle](n22) at (5.5,5)  {$\bullet$};
\node[scale=0.7, draw, circle](n23) at (9.5,5)  {$\bullet$};
\node[scale=0.7, draw, circle](n24) at (13,5)  {$\bullet$}; 
\node[scale=0.7, draw, circle](n25) at (3.5,5)  {$\bullet$};
\node[scale=0.7, draw, circle](n26) at (7.5,5)  {$\bullet$};
\node[scale=0.7, draw, circle](n27) at (11.5,5)  {$\bullet$};
\node[scale=0.7, draw, circle](n28) at (16,5)  {$\bullet$};
\node[scale=0.7, draw, circle](n29) at (14.5,5)  {$\bullet$};
\node[scale=0.7, draw, circle](n31) at (2.5,7)  {$\bullet \bullet$};
\node[scale=0.7, draw, circle](n32) at (6.5,7)  {$\bullet \bullet$};
\node[scale=0.7, draw, circle](n33) at (10.5,7)  {$\bullet \bullet$};
\node[scale=0.7, draw, circle](n34) at (14.5,7)  {$\bullet \bullet$};
\node[scale=0.7, draw, circle](n41) at (4.5,9)  {$\bullet \bullet$};
\node[scale=0.7, draw, circle](n42) at (12.5,9)  {$\bullet \bullet$};
\node[scale=0.7, draw, circle](n51) at (8.5,11)  {$\bullet \bullet$};

\foreach \from/\to in {n51/n41, n51/n42, n41/n31, n41/n32, n42/n33, n42/n34, n31/n21, n32/n22, n33/n23, n34/n24,  n25/n31, n26/n32, n27/n33, n28/n34, n34/n29}
\draw (\from) -- (\to);
\end{tikzpicture}
\end{center}
\caption{Secondary structure (left) and corresponding tree representation (right)}
\label{mainproblemnolabels}
\end{figure}
Nucleotide positions belonging to a base pair are represented in the tree by a two-dot vertex, while each unpaired nucleotide position is represented in the tree by a one-dot vertex. The children of each two-dot vertex $v$ are all the vertices corresponding to nucleotides nested directly beneath the arc corresponding to $v$ in the secondary structure. 

We consider only secondary structures for which each nucleotide in the sequence is attached to at most one arc, and which are \emph{pseudoknot-free}: 
arcs drawn above the nucleotide positions as in Figure~\ref{mainproblemnolabels},
representing base pairs, do not intersect (because arc crossings lead to complex constraints that often cannot be realized in three dimensions~\cite{yann}).
We shall assume that the first and last nucleotide of the sequence form a base pair, so that the graph representation is a rooted tree. (Otherwise it is a forest, in which case we can introduce a virtual root in order to produce a tree.) Each tree vertex will be labelled with either one or two dots;
the only vertices that can be labelled with one dot are the leaves.

Now suppose that we are given only the target secondary structure $R$ of Figure~\ref{mainproblemnolabels}. The combinatorial RNA design problem then asks whether the nucleotide positions of the secondary structure can be labelled so that the resulting nucleotide sequence admits a unique set of \{G,C\} and \{A,U\} arcs of maximum size and so that this arc set equals~$R$. We shall see in Theorem~\ref{main} that the answer for this example is yes, and that one such labelling is given by Figure~\ref{mainproblem}.

We formalize these ideas as follows. 

\begin{defn}[Secondary structure]
A (pseudoknot-free) \emph{secondary structure} applied to a sequence of nucleotide positions $\{1,2,\dots,n\}$ is a set of arcs $\{i,j\}$ satisfying $1 \le i < j \le n$, for which no two arcs share the same endpoint and there are no two arcs $\{i,j\}$ and $\{k, \ell\}$ satisfying $i < k < j < \ell$.
\end{defn}

\begin{defn}[Maximum-size arc set]
A secondary structure applied to a nucleotide sequence is a \emph{maximum-size arc set} for the sequence if it contains the largest possible number of arcs (where only arcs of types \{G,C\} and \{A,U\} are allowed).
\end{defn}

\begin{defn}[Design, designable]
An RNA sequence is a \emph{design} if it admits a unique maximum-size arc set.
A secondary structure $R$ is \emph{designable} if there is a design $S$ whose (unique) maximum-size arc set equals $R$; in that case, we say that $S$ is a design for~$R$.
\end{defn}

\noindent {\bf Combinatorial RNA Design Problem:} Given a target secondary structure $R$, find a design for $R$ or else show that $R$ is not designable.
\mbox{}\\

We introduce some further definitions in order to summarize prior results on the combinatorial RNA design problem. 
Following \cite{halevs2016combinatorial}, a secondary structure in which all nucleotide positions are paired is \emph{saturated}, and otherwise is \emph{unsaturated}. 
The tree corresponding to a secondary structure is saturated when every leaf has two dots, and otherwise is unsaturated.
\begin{figure}[H]
\centering
\vspace{-5mm}
\subfloat[]{
\begin{tikzpicture}
\node at (1,1) {$\bullet$};
\node at (1.5,1) {$\bullet$};
\node at (2,1) {$\bullet$};
\node at (2.5,1) {$\bullet$};
\node at (3,1) {$\bullet$};
\node at (3.5,1) {$\bullet$};
\node at (4,1) {$\bullet$};
\node at (4.5,1) {$\bullet$};
\node at (5,1) {$\bullet$};
\node at (5.5,1) {$\bullet$};
\draw[bend left=65] (1,1.25) to (5.5,1.25);
\draw[bend left=65] (1.5,1.25) to (4,1.25);
\draw[bend left=65] (4.5,1.25) to (5,1.25);
\draw[bend left=65] (2,1.25) to (2.5,1.25);
\draw[bend left=65] (3,1.25) to (3.5,1.25);
\end{tikzpicture}}\hspace{2cm}
\subfloat[]{
\begin{tikzpicture}
\node at (1,1) {$\bullet$};
\node at (1.5,1) {$\bullet$};
\node at (2.25,1) {$\bullet$};
\node at (3,1) {$\bullet$};
\node at (3.5,1) {$\bullet$};
\node at (4,1) {$\bullet$};
\node at (4.5,1) {$\bullet$};
\node at (5,1) {$\bullet$};
\node at (5.5,1) {$\bullet$};
\draw[bend left=65] (1,1.25) to (5.5,1.25);
\draw[bend left=65] (1.5,1.25) to (4,1.25);
\draw[bend left=65] (4.5,1.25) to (5,1.25);
\draw[bend left=65] (3,1.25) to (3.5,1.25);
\end{tikzpicture}}
\vspace{-5mm}
\caption[Saturated and unsaturated secondary structures]{Saturated (left) and unsaturated (right) secondary structures}
\end{figure}

\begin{thm}[{Hale{\v{s}} et al.\ \cite[Result~R4]{halevs2016combinatorial}}]
\label{old} 
The secondary structure $R$ corresponding to a rooted saturated tree $T$ is designable if and only if the root of $T$ has at most four children and every other vertex of $T$ has at most three children.
\end{thm}
\noindent
The proof of Theorem \ref{old} given in \cite{halevs2016combinatorial} is constructive, and gives many possible nucleotide labellings for the tree. 

\begin{thm}[{Hale{\v{s}} et al.\ \cite[Result~R5]{halevs2016combinatorial}}]
\label{tertiary} 
The secondary structure $R$ corresponding to a rooted unsaturated tree $T$ containing a vertex with at least two $2$-dot children and at least one $1$-dot child is not designable.
\end{thm}

Theorem~\ref{tertiary} shows that we cannot remove the constraint ``saturated'' from Theorem~\ref{old}: for example, the secondary structure corresponding to the unsaturated tree 
\begin{center}
\begin{tikzpicture}[scale=0.4]
\node[scale=0.7, draw, circle](n41) at (4.5,9)  {$\bullet \bullet$};
\node[scale=0.7, draw, circle](n42) at (12.5,9)  {$\bullet \bullet$};
\node[scale=0.7, draw, circle](n43) at (8.5,9)  {$\bullet$};
\node[scale=0.7, draw, circle](n51) at (8.5,11)  {$\bullet \bullet$};

\foreach \from/\to in {n51/n41, n51/n42, n51/n43}
\draw (\from) -- (\to);
\end{tikzpicture}
\end{center}
\noindent
is not designable.
Nonetheless, in this paper we generalize Theorem~\ref{old} by showing that an infinite class of unsaturated secondary structures is designable, thereby solving the combinatorial RNA design problem for a new class of structures.

Following \cite{halevs2016combinatorial}, a \emph{colouring} of a rooted tree assigns a colour drawn from the set \{B, W, Y\} (where B = black, W = white, Y = grey) to each paired vertex of a rooted tree except the root. The colouring is \emph{proper} if it obeys the rules specified in Table~\ref{colouring}.
(These rules constrain the root to have at most four children, and each other vertex to have at most three children.)
The \emph{level} of a vertex in such a coloured tree is the difference between the number of black vertices and the number of white vertices on the path starting at the vertex and ending at the root. A proper colouring is \emph{separated} if no grey vertex shares a level with an unpaired vertex.

\begin{table}[h]
\begin{center}
\begin{tabular}{c | c}
Vertex colouring 	& The multiset of colours of the children must be contained in the multiset\\ \hline
none (root)		& \{B, W, Y, Y\} \\
B			& \{B, Y, Y\}	\\
W			& \{W, Y, Y\}	\\
Y			& \{B, W, Y\}	
\end{tabular}
\end{center}
\caption{Rules for a proper colouring of the vertices of a rooted tree}
\label{colouring}
\end{table}

\begin{thm}[{Hale{\v{s}} et al.\ \cite[Result~R6]{halevs2016combinatorial}}]
\label{sepcol} 
The secondary structure $R$ corresponding to a rooted tree $T$ which admits a separated proper colouring is designable.
\end{thm}
\noindent
Theorem~\ref{sepcol} is proved in \cite{halevs2016combinatorial} using an algorithm for constructing a design for the secondary structure $R$ in which all leaves are labelled with A, and then the remaining black vertices are labelled with GC, white vertices with CG, and grey vertices with one of AU and~UA.

In Section~\ref{sec:results} we shall introduce the infinite class of ``P-unsaturated perfect floral trees'', and state without proof our main result (Theorem~\ref{main}) that these trees corresponds to designable secondary structures under the ``natural labelling'' of the structure with nucleotides.
We then show that our main result complements the previous result of Theorem~\ref{sepcol}, in that each can be used to establish the designability of a secondary structure that the other cannot. 
We then establish some extensions of the main result, as Corollaries \ref{prune} and \ref{GUcorollary}.
Verifying that the natural labelling gives a maximum-size arc set is straightforward, but proving that the corresponding sequence is a design for the target structure appears to be much more difficult and subtle.
In Section~\ref{sec:pics} we illustrate the proof techniques for Theorem~\ref{main} by reference to an extended example.
We then prove the main result in full generality in Section~\ref{sec:proof}.
We conclude in Section~\ref{sec:future} with some open questions.

\section{Main result and some extensions}
\label{sec:results}

We wish to identify a new class of unsaturated trees that are designable (and which must necessarily avoid the class of unsaturated trees described in Theorem~\ref{tertiary}). In order to do so, we introduce Definition~\ref{def:P-unsat}. Note that the \emph{depth} of a vertex $v$ in a rooted tree is the length of the path from the root to $v$. 

\begin{defn}[P-unsaturated]
\label{def:P-unsat}
A rooted tree is \emph{P-unsaturated} if every leaf at the maximum depth is assigned exactly one dot and all other vertices exactly two dots.
\end{defn}
\noindent
(The letter P in the name P-unsaturated is inspired by the requirement that all one-dot vertices of the rooted tree must be Pendant vertices.)

\begin{figure}[H]
\centering
\begin{tikzpicture}[scale=0.5]
\node[scale=0.7, draw, circle](n31) at (2.5,7)  {$\bullet$};
\node[scale=0.7, draw, circle](n32) at (6.5,7)  {$\bullet$};
\node[scale=0.7, draw, circle](n41) at (4.5,9)  {$\bullet \bullet$};
\node[scale=0.7, draw, circle](n42) at (12.5,9)  {$\bullet \bullet$};
\node[scale=0.7, draw, circle](n51) at (8.5,11)  {$\bullet \bullet$};

\foreach \from/\to in {n51/n41, n51/n42, n41/n31, n41/n32}
\draw (\from) -- (\to);
\end{tikzpicture}
\caption{P-unsaturated binary tree}
\end{figure}
\noindent
A \emph{binary tree} is a rooted tree in which every vertex has at most two children. 
A \emph{perfect binary tree} is a rooted tree in which every internal vertex has exactly two children and all leaves are at the same depth. 
We call a rooted tree \emph{floral} if removing all vertices at the maximum depth leaves a binary tree, and \emph{perfect floral} if removing all vertices at the maximum depth leaves a perfect binary tree. For example, Figure~\ref{Punfft} shows a P-unsaturated perfect floral tree of height~3.
\begin{figure}[h]
\centering
\begin{tikzpicture}[scale=0.6]
\node[scale=0.7, draw, circle](n21) at (1,5)  {$\bullet$};
\node[scale=0.7, draw, circle](n22) at (2,5)  {$\bullet$};
\node[scale=0.7, draw, circle](n23) at (3,5)  {$\bullet$};
\node[scale=0.7, draw, circle](n24) at (4,5)  {$\bullet$}; 
\node[scale=0.7, draw, circle](n25) at (6,5)  {$\bullet$}; 
\node[scale=0.7, draw, circle](n255) at (7,5)  {$\bullet$};
\node[scale=0.7, draw, circle](n266) at (14.5,5)  {$\bullet$};
\node[scale=0.7, draw, circle](n31) at (2.5,7)  {$\bullet\bullet$};
\node[scale=0.7, draw, circle](n32) at (6.5,7)  {$\bullet\bullet$};
\node[scale=0.7, draw, circle](n33) at (10.5,7)  {$\bullet\bullet$};
\node[scale=0.7, draw, circle](n34) at (14.5,7)  {$\bullet\bullet$};
\node[scale=0.7, draw, circle](n41) at (4.5,9)  {$\bullet\bullet$};
\node[scale=0.7, draw, circle](n42) at (12.5,9)  {$\bullet\bullet$};
\node[scale=0.7, draw, circle](n51) at (8.5,11)  {$\bullet\bullet$};

\foreach \from/\to in {n51/n41, n51/n42, n41/n31, n41/n32, n42/n33, n42/n34, n31/n21, n31/n22, n31/n23, n31/n24, n32/n25, n32/n255, n34/n266}
\draw (\from) -- (\to);
\end{tikzpicture}
\caption{P-unsaturated perfect floral tree}\label{perfectfloral}
\label{Punfft}
\end{figure}

Theorem~\ref{main} (whose full proof is postponed until Section~\ref{sec:proof}) states that the secondary structure correponding to a P-unsaturated perfect floral tree is designable. The statement of this result involves the following assignment of nucleotides.

\begin{defn}[Natural labelling]
\label{def:natural}
Let $T$ be a P-unsaturated perfect floral tree of height $n+1$. The \emph{natural labelling} of $T$ labels all vertices at depth $n+1$ with A if $n+1$ is even, and with C if $n+1$ is odd; and labels the remaining vertices at even depths from left to right with alternating AU and UA, and those at odd depths from left to right with alternating GC and CG. 
\end{defn}
\noindent
For example, the natural labelling of the tree in Figure~\ref{perfectfloral} is shown in Figure~\ref{natlab}.

\begin{figure}[h]
\centering
\begin{tikzpicture}[scale=0.6]
\node[scale=0.7, draw, circle](n21) at (1,5)  {C};
\node[scale=0.7, draw, circle](n22) at (2,5)  {C};
\node[scale=0.7, draw, circle](n23) at (3,5)  {C};
\node[scale=0.7, draw, circle](n24) at (4,5)  {C}; 
\node[scale=0.7, draw, circle](n25) at (6,5)  {C};
\node[scale=0.7, draw, circle](n255) at (7,5)  {C}; 
\node[scale=0.7, draw, circle](n266) at (14.5,5)  {C};
\node[scale=0.7, draw, circle](n31) at (2.5,7)  {AU};
\node[scale=0.7, draw, circle](n32) at (6.5,7)  {UA};
\node[scale=0.7, draw, circle](n33) at (10.5,7)  {AU};
\node[scale=0.7, draw, circle](n34) at (14.5,7)  {UA};
\node[scale=0.7, draw, circle](n41) at (4.5,9)  {GC};
\node[scale=0.7, draw, circle](n42) at (12.5,9)  {CG};
\node[scale=0.7, draw, circle](n51) at (8.5,11)  {AU};

\foreach \from/\to in {n51/n41, n51/n42, n41/n31, n41/n32, n42/n33, n42/n34, n31/n21, n31/n22, n31/n23, n31/n24, n32/n25, n32/n255, n34/n266}
\draw (\from) -- (\to);
\end{tikzpicture}
\caption{Natural labelling of P-unsaturated perfect floral tree}\label{labelledperfectfloral}
\label{natlab}
\end{figure}

\begin{thm}[Main Result]\label{main}
The secondary structure $R$ corresponding to a P-unsaturated perfect floral tree $T$ is designable, and a design for $R$ is given by the nucleotide sequence $S$ corresponding to the natural labelling of~$T$.
\end{thm}

Theorems~\ref{sepcol} and~\ref{main} complement each other. 
The unsaturated tree shown in Figure~\ref{separated}, for example, admits the displayed separated colouring and so is designable by Theorem~\ref{sepcol}, but is not floral and so cannot be shown to be designable using Theorem~\ref{main} (nor by its Corollary~\ref{prune} below).
Conversely, consider a P-unsaturated perfect floral tree $T$ of height~5, each of whose vertices at depth 4 has one child that is assigned one dot. We know this tree is designable by Theorem~\ref{main}, but we now demonstrate with reference to Figure~\ref{notseparated} that it cannot be shown to be designable using Theorem~\ref{sepcol}. Suppose, for a contradiction, that we can assign a separated proper colouring to~$T$. By the separated colouring condition, no vertex at depth~4 can be coloured~Y (else it would share a level with its unpaired child), and then by the colouring rules given in Table~\ref{colouring} each pair of sibling vertices at depth~4 must be coloured as one B and one~W. The colouring rules applied to vertices at depth 3, then depth 2, and so on, then force
all vertices at depths 1 and 3 to be coloured~Y, and
all vertices at depth 2 to be coloured in sibling pairs as one B and one~W. But then the grey vertices at depth 1 share the level~0 with the shaded unpaired vertex contained in the subtree of $T$ shown in Figure~\ref{notseparated}, contradicting that the proper colouring is separated.

\begin{figure}[h]
\centering
\begin{tikzpicture}[scale=0.25]
\node[scale=0.7, draw, circle](n01) at (15,14)  {$\bullet \bullet$};
\node[scale=0.7, draw, circle](n11) at ( 0,8)  {$\bullet \bullet$};
\node[scale=0.7, draw, circle](n12) at (10,8)  {$\bullet \bullet$};
\node[scale=0.7, draw, circle](n13) at (20,8)  {$\bullet \bullet$};
\node[scale=0.7, draw, circle](n14) at (30,8)  {$\bullet \bullet$};
\node[] at (-1, 10)  {B};
\node[] at ( 9, 10)  {W};
\node[] at (20, 10)  {Y};
\node[] at (30, 10)  {Y};
\node[] at (-1, 6)  {1};
\node[] at ( 9, 6)  {$-1$};
\node[] at (20, 6)  {0};
\node[] at (30, 6)  {0};
\node[scale=0.7, draw, circle](n21) at (0,2)  {$\bullet$};
\node[scale=0.7, draw, circle](n22) at (10,2)  {$\bullet$};
\node[] at (-1, 0)  {1};
\node[] at ( 9, 0)  {$-1$};

\foreach \from/\to in {
n01/n11, n01/n12, n01/n13, n01/n14,
n11/n21, n12/n22}
\draw (\from) -- (\to);
\end{tikzpicture}

\caption{Unsaturated tree admitting a separated proper colouring (levels shown beneath vertices)}
\label{separated}
\end{figure}

\begin{figure}[h]
\centering
\begin{tikzpicture}[scale=0.17]
\node[scale=0.7, draw, circle](n01) at (32,30)  {$\bullet \bullet$};
\node[scale=0.7, draw, circle](n11) at (16,24)  {$\bullet \bullet$};
\node[scale=0.7, draw, circle](n12) at (48,24)  {$\bullet \bullet$};
\node[] at (16, 26.5)  {Y};
\node[] at (48, 26.5)  {Y};
\node[scale=0.7, draw, circle](n21) at (8,18)  {$\bullet \bullet$};
\node[scale=0.7, draw, circle](n22) at (24,18)  {$\bullet \bullet$};
\node[] at (8, 20.5)   {B};
\node[] at (24, 20.5)  {W};
\node[scale=0.7, draw, circle](n31) at (4,12)  {$\bullet \bullet$};
\node[scale=0.7, draw, circle](n32) at (12,12)  {$\bullet \bullet$};
\node[] at (4, 14.5)   {Y};
\node[] at (12, 14.5)  {Y};
\node[scale=0.7, draw, circle](n41) at (2,6)  {$\bullet \bullet$};
\node[scale=0.7, draw, circle](n42) at (6,6)  {$\bullet \bullet$};
\node[] at (1.8, 8.5)  {B};
\node[] at (6.2, 8.5)  {W};
\node[scale=0.7, draw, circle](n51) at (2,0)  {$\bullet$};
\node[fill=red!30, scale=0.7, draw, circle](n52) at (6,0)  {$\bullet$};

\foreach \from/\to in {
n01/n11, n01/n12, 
n11/n21, n11/n22, 
n21/n31, n21/n32, 
n31/n41, n31/n42, 
n41/n51, n42/n52}
\draw (\from) -- (\to);

\draw[dash pattern=on 2pt off 1pt] (n12) -- (44,21);
\draw[dash pattern=on 2pt off 1pt] (n12) -- (52,21);
\draw[dash pattern=on 2pt off 1pt] (n22) -- (22,15);
\draw[dash pattern=on 2pt off 1pt] (n22) -- (26,15);
\draw[dash pattern=on 2pt off 1pt] (n32) -- (11,9);
\draw[dash pattern=on 2pt off 1pt] (n32) -- (13,9);
\end{tikzpicture}

\caption{Subtree of proper colouring of P-unsaturated perfect floral tree of height~5}
\label{notseparated}
\end{figure}

We show in Corollary~\ref{prune} that we can relax the condition ``perfect floral'' in Theorem~\ref{main} to ``floral''. The example structures in Figure~\ref{prune-illustration} illustrate the ideas of the proof. 

\begin{cor}
\label{prune}
The secondary structure $R$ corresponding to a P-unsaturated floral tree $T$ is designable, and a design for $R$ is given by the nucleotide sequence corresponding to a labelled subtree of the natural labelling of a perfect floral tree.
\end{cor}

\begin{proof}
Let $T'$ be the smallest perfect floral tree containing $T$ as a subtree; $T'$ is necessarily P-unsaturated. Let $\tau'$ be the labelled perfect floral tree obtained by assigning nucleotides to $T'$ according to the natural labelling. By Theorem \ref{main}, the labelled tree $\tau'$ yields a design $S'$ of nucleotides whose (unique) maximum-size arc set corresponds to $T'$. Let $\tau$ be the labelled subtree of $\tau'$ whose unlabelled version is $T$, and let $S$ be the associated sequence of nucleotides.

Colour the nucleotides and arcs that belong to $S'$ but not to~$S$ red.
The red substructure of $S'$ consists of all nucleotides and corresponding arcs that must be pruned from the labelled perfect floral tree $\tau'$ to obtain the labelled floral tree~$\tau$.
It follows that the secondary structure $R$ corresponding to $T$ is a maximum-size arc set for~$S$: if $S$ admits a larger arc set than $R$ then that arc set, together with the red arc set of $S'$, gives a larger arc set for $S'$ (contradicting that the secondary structure corresponding to $T'$ is a maximum-size arc set for $S'$).
It also follows that $S$ is a design for $R$: if $S$ admits an alternative maximum-size arc set to $R$ then that arc set, together with the red arc set of $S'$, gives an alternative maximum-size arc set for $S'$ (contradicting that $S'$ is a design).
\end{proof}

\begin{figure}[h]
\centering
\begin{tikzpicture}[scale=0.42]
\node[scale=0.7, draw, circle](n022) at (6.5,5)  {$\bullet$};
\node[scale=0.7, draw, circle](n023) at (5,5)  {$\bullet$};
\node[scale=0.7, draw, circle](n024) at (8,5)  {$\bullet$};
\node[scale=0.7, draw, circle](n032) at (6.5,7)  {$\bullet \bullet$};
\node[scale=0.7, draw, circle](n041) at (4.5,9)  {$\bullet \bullet$};
\node[scale=0.7, draw, circle](n042) at (12.5,9)  {$\bullet \bullet$};
\node[scale=0.7, draw, circle](n051) at (8.5,11)  {$\bullet \bullet$};

\node[scale=1](n052) at (2.5,11)  {$T$:};

\foreach \from/\to in {n051/n041, n051/n042, n041/n032, n032/n022, n032/n023, n032/n024}
\draw (\from) -- (\to);
\end{tikzpicture}
\hspace{1.5cm}
\begin{tikzpicture}[scale=0.42]
\node[scale=0.7, draw, circle](n000022) at (6.5,5)  {$\bullet$};
\node[scale=0.7, draw, circle](n000023) at (5,5)  {$\bullet$};
\node[scale=0.7, draw, circle](n000024) at (8,5)  {$\bullet$};
\node[scale=0.7, draw, circle, red](n31) at (2.5,7)  {$\bullet \bullet$};
\node[scale=0.7, draw, circle](n32) at (6.5,7)  {$\bullet \bullet$};
\node[scale=0.7, draw, circle, red](n33) at (10.5,7)  {$\bullet \bullet$};
\node[scale=0.7, draw, circle, red](n34) at (14.5,7)  {$\bullet \bullet$};
\node[scale=0.7, draw, circle](n41) at (4.5,9)  {$\bullet \bullet$};
\node[scale=0.7, draw, circle](n42) at (12.5,9)  {$\bullet \bullet$};
\node[scale=0.7, draw, circle](n51) at (8.5,11)  {$\bullet \bullet$};

\node[scale=1](n052) at (2.5,11)  {$T'$:};

\foreach \from/\to in {n51/n41, n51/n42, n41/n32, n32/n000024, n32/n000023, n32/n000022}
\draw (\from) -- (\to);
\draw[red] (n41) -- (n31);
\draw[red] (n42) -- (n33);
\draw[red] (n42) -- (n34);
\end{tikzpicture}

\vspace{0.5cm}
\begin{tikzpicture}[scale=0.42]
\node[scale=0.7, draw, circle](n22) at (6.5,5)  {C};
\node[scale=0.7, draw, circle](n23) at (5,5)  {C};
\node[scale=0.7, draw, circle](n24) at (8,5)  {C};
\node[scale=0.7, draw, circle](n32) at (6.5,7)  {UA};
\node[scale=0.7, draw, circle](n41) at (4.5,9)  {GC};
\node[scale=0.7, draw, circle](n42) at (12.5,9)  {CG};
\node[scale=0.7, draw, circle](n51) at (8.5,11)  {AU};

\node[scale=1](n052) at (2.5,11)  {$\tau$:};

\foreach \from/\to in {n51/n41, n51/n42, n41/n32, n32/n22, n32/n23, n32/n24}
\draw (\from) -- (\to);
\end{tikzpicture}
\hspace{1.5cm}
\begin{tikzpicture}[scale=0.42]
\node[scale=0.7, draw, circle](n00022) at (6.5,5)  {C};
\node[scale=0.7, draw, circle](n00023) at (5,5)  {C};
\node[scale=0.7, draw, circle](n00024) at (8,5)  {C};
\node[scale=0.7, draw, circle, red](n31) at (2.5,7)  {AU};
\node[scale=0.7, draw, circle](n32) at (6.5,7)  {UA};
\node[scale=0.7, draw, circle, red](n33) at (10.5,7)  {AU};
\node[scale=0.7, draw, circle, red](n34) at (14.5,7)  {UA};
\node[scale=0.7, draw, circle](n41) at (4.5,9)  {GC};
\node[scale=0.7, draw, circle](n42) at (12.5,9)  {CG};
\node[scale=0.7, draw, circle](n51) at (8.5,11)  {AU};

\node[scale=1](n052) at (2.5,11)  {$\tau'$:};

\foreach \from/\to in {n51/n41, n51/n42, n41/n32, n32/n00024, n32/n00023, n32/n00022}
\draw (\from) -- (\to);
\draw[red] (n41) -- (n31);
\draw[red] (n42) -- (n33);
\draw[red] (n42) -- (n34);
\end{tikzpicture}

\vspace{0.2cm}
\begin{tikzpicture}[scale=0.8]
\node at (0.25,1) {$S$:};
\node at (1,1) {A};
\node at (1.5,1) {G};
\node at (2,1) {U};
\node at (2.5,1) {C};
\node at (3,1) {C};
\node at (3.5,1) {C};
\node at (4,1) {A};
\node at (4.5,1) {C};
\node at (5,1) {C};
\node at (5.5,1) {G};
\node at (6,1) {U};
\draw[bend left=65] (1,1.25) to (6,1.25);
\draw[bend left=65] (1.5,1.25) to (4.5,1.25);
\draw[bend left=65] (5,1.25) to (5.5,1.25);
\draw[bend left=65] (2,1.25) to (4,1.25);
\end{tikzpicture}\hspace{1.5cm}
\begin{tikzpicture}[scale=0.8]
\node at (0.25,1) {$S'$:};
\node at (1,1) {A};
\node at (1.5,1) {G};
\node[red] at (2,1) {A};
\node[red] at (2.5,1) {U};
\node at (3,1) {U};
\node at (3.5,1) {C};
\node at (4,1) {C};
\node at (4.5,1) {C};
\node at (5,1) {A};
\node at (5.5,1) {C};
\node at (6,1) {C};
\node[red] at (6.5,1) {A};
\node[red] at (7,1) {U};
\node[red] at (7.5,1) {U};
\node[red] at (8,1) {A};
\node at (8.5,1) {G};
\node at (9,1) {U};
\draw[bend left=65] (1,1.25) to (9,1.25);
\draw[bend left=65] (1.5,1.25) to (5.5,1.25);
\draw[bend left=65] (6,1.25) to (8.5,1.25);
\draw[bend left=65, red] (2,1.25) to (2.5,1.25);
\draw[bend left=65] (3,1.25) to (5,1.25);
\draw[bend left=65, red] (6.5,1.25) to (7,1.25);
\draw[bend left=65, red] (7.5,1.25) to (8,1.25);
\end{tikzpicture}

\caption{Illustration of the proof of Corollary~\ref{prune}}
\label{prune-illustration}
\end{figure}

Recall from Section~\ref{sec:intro} that we have assumed the Watson-Crick model  (under which minimizing the total energy is equivalent to maximizing the number of \{G,C\} and \{A,U\} base pairs). We show in Corollary~\ref{GUcorollary} that the result of Corollary~\ref{prune} carries over to the Nussinov-Jacobson model with $\alpha=\beta=\gamma=-1$, under which minimizing the total energy is equivalent to maximizing the number of \{G,C\} and \{A,U\} and \{G,U\} base pairs.

\begin{cor}
\label{GUcorollary}
The secondary structure $R$ corresponding to a P-unsaturated floral tree $T$ is designable under the  model in which minimizing the total energy is equivalent to maximizing the number of \{G,C\} and \{A,U\} and \{G,U\} base pairs.
\end{cor}

\begin{proof}
Let $R$ be the secondary structure corresponding to a P-unsaturated floral tree~$T$. By Corollary~\ref{prune}, there is a design $S$ for $R$ under the Watson-Crick model (containing no \{G,U\} base pairs), and $S$ corresponds to a labelled subtree of the natural labelling of a perfect floral tree. 
Since all G and U nucleotides are paired in the natural labelling, the number of arcs in $R$ equals the number $g$ of $G$ nucleotides in $S$ plus the number $u$ of $U$ nucleotides in~$S$.

We shall show that $S$ is also a design for $R$ under the model which involves maximizing the number of \{G,C\} and \{A,U\} and \{G,U\} base pairs, by showing that an alternative arc set to $R$ that includes at least one \{G,U\} base pair must contain fewer arcs than the $g+u$ contained in~$R$. Let the number of \{G,U\} base pairs in an arc set under this more general model be $n>0$. Then there are
at most $g-n$ nucleotides G remaining to form \{G,C\} base pairs, and 
at most $u-n$ nucleotides U remaining to form \{A,U\} base pairs, giving a total base pair count of at most $(g-n)+(u-n)+n < g+u$.
\end{proof}

\section{Illustration of proof techniques}
\label{sec:pics}

Although the natural labelling of a P-unsaturated perfect floral tree follows a simple pattern, and it is relatively straightforward to show that Theorem \ref{main} holds when $T$ has small height, finding a general argument that applies to all tree heights appears to be very delicate.
In Section~\ref{sec:proof} we will prove Theorem~\ref{main} according to the following outline 
(where $R$ is the secondary structure corresponding to a P-unsaturated perfect floral tree~$T$, and $S$ is the nucleotide sequences corresponding to the natural labelling of~$T$).
Let $R$, $S$, and $T$ be as given in Theorem~\ref{main}.
We must establish that~$S$ has a unique maximum-size arc set, and that this arc set equals~$R$.
We first show that all nucleotides in $S$ that are unpaired in $T$ (those at depth $n+1$) must also remain unpaired in every maximum-size arc set. In view of this, we show how to reduce to the case of a saturated perfect binary tree. Our main tool is a running difference of strategically chosen subsets of the nucleotides in $S$, and a key insight is to focus on the parity of these running differences.

Having reduced a P-unsaturated perfect floral tree to a saturated perfect binary tree, we could then apply Theorem~\ref{old} to show that the secondary structure is designable. However, to keep our proof self-contained and to demonstrate the versatility of the method of running differences, we instead give our own proof in which we show that for every depth~$j$ of the reduced tree, nucleotides at depth~$j$ must all pair with one another, and that this forces a maximum-size arc set for $S$ to equal~$R$.

We now give a detailed illustration of the proof of Theorem \ref{main}, using the example of a P-unsaturated perfect binary tree $T$ of height 4. (We use a binary tree in this illustration simply for convenience: the arguments readily extend to a P-unsaturated perfect floral tree.) The natural labelling $\tau$ of $T$ is:

\begin{center}
\begin{tikzpicture}
[scale=0.475,auto=left]
\node[circle, scale=0.7, draw](n21) at (1.5,5)  {A};
\node[circle, scale=0.7, draw](n22) at (3.5,5)  {A};
\node[circle, scale=0.7, draw](n23) at (5.5,5)  {A};
\node[circle, scale=0.7, draw](n24) at (7.5,5)  {A};
\node[circle, scale=0.7, draw](n25) at (9.5,5)  {A};
\node[circle, scale=0.7, draw](n26) at (11.5,5)  {A};
\node[circle, scale=0.7, draw](n27) at (13.5,5)  {A};
\node[circle, scale=0.7, draw](n28) at (15.5,5)  {A};
\node[circle, scale=0.7, draw](n29) at (17.5,5)  {A};
\node[circle, scale=0.7, draw](n210) at (19.5,5)  {A};
\node[circle, scale=0.7, draw](n211) at (21.5,5)  {A};
\node[circle, scale=0.7, draw](n212) at (23.5,5)  {A};
\node[circle, scale=0.7, draw](n213) at (25.5,5)  {A};
\node[circle, scale=0.7, draw](n214) at (27.5,5)  {A};
\node[circle, scale=0.7, draw](n215) at (29.5,5)  {A};
\node[circle, scale=0.7, draw](n216) at (31.5,5)  {A};
\node[circle, scale=0.7, draw](n31) at (2.5,7)  {GC};
\node[circle, scale=0.7, draw](n32) at (6.5,7)  {CG};
\node[circle, scale=0.7, draw](n33) at (10.5,7)  {GC};
\node[circle, scale=0.7, draw](n34) at (14.5,7)  {CG};
\node[circle, scale=0.7, draw](n35) at (18.5,7)  {GC};
\node[circle, scale=0.7, draw](n36) at (22.5,7)  {CG};
\node[circle, scale=0.7, draw](n37) at (26.5,7)  {GC};
\node[circle, scale=0.7, draw](n38) at (30.5,7)  {CG};
\node[circle, scale=0.7, draw](n41) at (4.5,9)  {AU};
\node[circle, scale=0.7, draw](n42) at (12.5,9)  {UA};
\node[circle, scale=0.7, draw](n43) at (20.5,9)  {AU};
\node[circle, scale=0.7, draw](n44) at (28.5,9)  {UA};
\node[circle, scale=0.7, draw](n51) at (8.5,11)  {GC};
\node[circle, scale=0.7, draw](n52) at (24.5,11)  {CG};
\node[circle, scale=0.7, draw](n61) at (16.5,13)  {AU};

\foreach \from/\to in {n61/n51, n61/n52, n51/n41, n51/n42, n52/n43, n52/n44, n41/n31, n41/n32, n42/n33, n42/n34, n43/n35, n43/n36, n44/n37, n44/n38, n31/n21, n31/n22, n32/n23, n32/n24, n33/n25, n33/n26, n34/n27, n34/n28, n35/n29, n35/n210, n36/n211, n36/n212, n37/n213, n37/n214, n38/n215, n38/n216}
\draw (\from) -- (\to);
\end{tikzpicture}
\end{center}

The corresponding sequence $S$ (represented on two lines to assist reading) is:

\begin{center}
\begin{tikzpicture}  [scale=0.625,auto=left]
\node at (2,1) {A};
\node at (3,1) {G};
\node at (4,1) {A};
\node at (5,1) {G};
\node at (6,1) {A};
\node at (7,1) {A};
\node at (8,1) {C};
\node at (9,1) {C};
\node at (10,1) {A};
\node at (11,1) {A};
\node at (12,1) {G};
\node at (13,1) {U};
\node at (14,1) {U};
\node at (15,1) {G};
\node at (16,1) {A};
\node at (17,1) {A};
\node at (18,1) {C}; 
\node at (19,1) {C};
\node at (20,1) {A};
\node at (21,1) {A};
\node at (22,1) {G};
\node at (23,1) {A};
\node at (24,1) {C};
\node at (3,0) {C};
\node at (4,0) {A};
\node at (5,0) {G};
\node at (6,0) {A};
\node at (7,0) {A};
\node at (8,0) {C};
\node at (9,0) {C};
\node at (10,0) {A};
\node at (11,0) {A};
\node at (12,0) {G};
\node at (13,0) {U};
\node at (14,0) {U};
\node at (15,0) {G};
\node at (16,0) {A};
\node at (17,0) {A}; 
\node at (18,0) {C};
\node at (19,0) {C};
\node at (20,0) {A};
\node at (21,0) {A};
\node at (22,0) {G};
\node at (23,0) {A};
\node at (24,0) {G};
\node at (25,0) {U};
\end{tikzpicture}
\end{center}

We shall show that $S$ admits a unique maximum-size arc set, and that this arc set equals~$R$. We introduce the following definition.

\begin{defn}[Balanced set]
\label{balanced_defn}
A subset of nucleotides from a nucleotide sequence forms a \emph{balanced set} if it comprises either an equal number of Gs and Cs or an equal number of As and Us.
\end{defn}

For example, in Figure \ref{runningdifference} the subset of G and C nucleotides shown boxed and coloured forms a balanced set $\mathcal{B}$ within the nucleotide sequence $S$.
The tag shown in Figure~\ref{runningdifference} above each nucleotide~$N$ (not in $\mathcal{B}$) is the difference between the number of Gs and the number of Cs occurring in $\mathcal{B}$ to the left of $N$, namely the \emph{running difference} with respect to $\mathcal{B}$. We use the corresponding definition for the running difference with respect to a balanced set of A and U nucleotides.

\begin{figure}[H]
\centering
\begin{tikzpicture}  [scale=0.625,auto=left]
\node at (2,2) {A};
\node at (3,2)[draw, red] {G};
\node at (4,2) {A};
\node at (5,2)[draw,red] {G};
\node at (6,2) {A};
\node at (7,2) {A};
\node at (8,2)[draw,red] {C};
\node at (9,2)[draw,red] {C};
\node at (10,2) {A};
\node at (11,2) {A};
\node at (12,2)[draw,red] {G};
\node at (13,2) {U};
\node at (14,2) {U};
\node at (15,2)[draw,red] {G};
\node at (16,2) {A};
\node at (17,2) {A};
\node at (18,2)[draw,red] {C}; 
\node at (19,2)[draw,red] {C};
\node at (20,2) {A};
\node at (21,2) {A};
\node at (22,2)[draw,red] {G};
\node at (23,2) {A};
\node at (24,2)[draw,red] {C};
\node at (3,0)[draw,red] {C};
\node at (4,0) {A};
\node at (5,0)[draw,red] {G};
\node at (6,0) {A};
\node at (7,0) {A};
\node at (8,0)[draw,red] {C};
\node at (9,0)[draw,red] {C};
\node at (10,0) {A};
\node at (11,0) {A};
\node at (12,0)[draw,red] {G};
\node at (13,0) {U};
\node at (14,0) {U};
\node at (15,0)[draw,red] {G};
\node at (16,0) {A};
\node at (17,0) {A}; 
\node at (18,0)[draw,red] {C};
\node at (19,0)[draw,red] {C};
\node at (20,0) {A};
\node at (21,0) {A};
\node at (22,0)[draw,red] {G};
\node at (23,0) {A};
\node at (24,0)[draw,red] {G};
\node at (25,0) {U};
\node at (2,3) {0};
\node at (3,3)[red] {};
\node at (4,3) {1};
\node at (5,3)[red] {};
\node at (6,3) {2};
\node at (7,3) {2};
\node at (8,3)[red] {};
\node at (9,3)[red] {};
\node at (10,3) {0};
\node at (11,3) {0};
\node at (12,3)[red] {};
\node at (13,3) {1};
\node at (14,3) {1};
\node at (15,3)[red] {};
\node at (16,3) {2};
\node at (17,3) {2};
\node at (18,3)[red] {}; 
\node at (19,3)[red] {};
\node at (20,3) {0};
\node at (21,3) {0};
\node at (22,3)[red] {};
\node at (23,3) {1};
\node at (24,3)[red] {};
\node at (3,1)[red] {};
\node at (4,1) {-1};
\node at (5,1)[red] {};
\node at (6,1) {0};
\node at (7,1) {0};
\node at (8,1)[red] {};
\node at (9,1)[red] {};
\node at (10,1) {-2};
\node at (11,1) {-2};
\node at (12,1)[red] {};
\node at (13,1) {-1};
\node at (14,1) {-1};
\node at (15,1)[red] {};
\node at (16,1) {0};
\node at (17,1) {0}; 
\node at (18,1)[red] {};
\node at (19,1)[red] {};
\node at (20,1) {-2};
\node at (21,1) {-2};
\node at (22,1)[red] {};
\node at (23,1) {-1};
\node at (24,1)[red] {};
\node at (25,1) {0};
\end{tikzpicture}
\caption{A balanced set of nucleotides and its corresponding running differences}
\label{runningdifference}
\end{figure}

In order to prove Theorem \ref{main} for $T$, we iteratively define a succession of balanced sets of nucleotides, occurring as disjoint subsets of the sequence $S$, and use the running difference with respect to these balanced sets to constrain the possible arcs of a maximum-size arc set.

Suppose a maximum-size arc set $M$ is applied to the sequence $S$. The labelled tree $\tau$ from which $S$ is derived shows that $S$ admits an arc set for which only A nucleotides are unpaired; therefore in $M$ all G, C, and U nucleotides of the sequence must be paired.\\

\emph{Iteration 1:}

$S$ has an equal number of Gs and Cs, forming a balanced set $\mathcal{B}$ which is boxed and coloured as shown in Figure \ref{runningdifference}. To achieve the maximum number of arcs, all elements of $\mathcal{B}$ must pair with one another; this is indicated by the colouring, which will be retained in all subsequent iterations (whereas the boxes are a temporary notation used only in the current iteration). To avoid inducing an arc crossing, every arc in $M$ joining uncoloured nucleotides (As and Us) must therefore enclose equally many of the two types of boxed nucleotides (Gs and Cs). Tag the uncoloured nucleotides with the running difference with respect to $\mathcal{B}$, as in Figure \ref{runningdifference}.

Then the uncoloured nucleotides with a given tag cannot pair with uncoloured nucleotides with a different tag, and cannot pair with coloured nucleotides (because those must all pair with one another). Therefore, the uncoloured nucleotides with a given tag can pair only with one another. This implies the weaker result that uncoloured nucleotides whose tags have the same parity can pair only with one another. (We have not yet shown that the uncoloured nucleotides with a given tag must all pair with one another, nor what the pairings must be.)

As a visual aid, we copy the boxing, colouring, and tagging applied to the nucleotides of~$S$ onto the corresponding labelled tree $\tau$: the nucleotides belonging to $\mathcal{B}$ (the Gs and~Cs) are boxed and coloured, and the uncoloured nucleotides (the As and Us) are tagged as above. (This does not assume the result we wish to prove, namely that applying a maximum-size arc set to $S$ results in the labelled tree $\tau$.)

\begin{center}
\begin{tikzpicture}
[scale=0.475,auto=left]
\node[circle, scale=0.7, draw](n21) at (1.5,4)  {A};
\node[circle, scale=0.7, draw](n22) at (3.5,4)  {A};
\node[circle, scale=0.7, draw](n23) at (5.5,4)  {A};
\node[circle, scale=0.7, draw](n24) at (7.5,4)  {A};
\node[circle, scale=0.7, draw](n25) at (9.5,4)  {A};
\node[circle, scale=0.7, draw](n26) at (11.5,4)  {A};
\node[circle, scale=0.7, draw](n27) at (13.5,4)  {A};
\node[circle, scale=0.7, draw](n28) at (15.5,4)  {A};
\node[circle, scale=0.7, draw](n29) at (17.5,4)  {A};
\node[circle, scale=0.7, draw](n210) at (19.5,4)  {A};
\node[circle, scale=0.7, draw](n211) at (21.5,4)  {A};
\node[circle, scale=0.7, draw](n212) at (23.5,4)  {A};
\node[circle, scale=0.7, draw](n213) at (25.5,4)  {A};
\node[circle, scale=0.7, draw](n214) at (27.5,4)  {A};
\node[circle, scale=0.7, draw](n215) at (29.5,4)  {A};
\node[circle, scale=0.7, draw](n216) at (31.5,4)  {A};
\node[] at (1.5,5)  {2};
\node[] at (3.5,5)  {2};
\node[] at (5.5,5)  {0};
\node[] at (7.5,5)  {0};
\node[] at (9.5,5)  {2};
\node[] at (11.5,5)  {2};
\node[] at (13.5,5)  {0};
\node[] at (15.5,5)  {0};
\node[] at (17.5,5)  {0};
\node[] at (19.5,5)  {0};
\node[] at (21.5,5)  {-2};
\node[] at (23.5,5)  {-2};
\node[] at (25.5,5)  {0};
\node[] at (27.5,5)  {0};
\node[] at (29.5,5)  {-2};
\node[] at (31.5,5)  {-2};

\node[scale=0.7, draw,fill=red!30](n31) at (2.5,7)  {GC};
\node[scale=0.7, draw,fill=red!30](n32) at (6.5,7)  {CG};
\node[scale=0.7, draw, fill=red!30](n33) at (10.5,7)  {GC};
\node[scale=0.7, draw, fill=red!30](n34) at (14.5,7)  {CG};
\node[scale=0.7, draw, fill=red!30](n35) at (18.5,7)  {GC};
\node[scale=0.7, draw, fill=red!30](n36) at (22.5,7)  {CG};
\node[scale=0.7, draw, fill=red!30](n37) at (26.5,7)  {GC};
\node[scale=0.7, draw, fill=red!30](n38) at (30.5,7)  {CG};
\node[circle, scale=0.7, draw](n41) at (4.5,9)  {AU};
\node[circle, scale=0.7, draw](n42) at (12.5,9)  {UA};
\node[circle, scale=0.7, draw](n43) at (20.5,9)  {AU};
\node[circle, scale=0.7, draw](n44) at (28.5,9)  {UA};
\node[] at (4.5,10)  {1, 1};
\node[] at (12.5,10)  {1, 1};
\node[] at (20.5,10)  {-1, -1};
\node[] at (28.5,10)  {-1, -1};
\node[scale=0.7, draw, fill=red!30](n51) at (8.5,11)  {GC};
\node[scale=0.7, draw, fill=red!30](n52) at (24.5,11)  {CG};
\node[circle, scale=0.7, draw](n61) at (16.5,13)  {AU};
\node[] at (16.5,14)  {0, 0};

\foreach \from/\to in {n61/n51, n61/n52, n51/n41, n51/n42, n52/n43, n52/n44, n41/n31, n41/n32, n42/n33, n42/n34, n43/n35, n43/n36, n44/n37, n44/n38, n31/n21, n31/n22, n32/n23, n32/n24, n33/n25, n33/n26, n34/n27, n34/n28, n35/n29, n35/n210, n36/n211, n36/n212, n37/n213, n37/n214, n38/n215, n38/n216}
\draw (\from) -- (\to);
\end{tikzpicture}
\end{center}

As described, all uncoloured nucleotides whose tags are odd (this being the opposite parity to that of the nucleotides at depth 4) can pair only with one another. But these nucleotides form a new balanced set $\mathcal{B'}$ (comprising A and U nucleotides), and so must all pair with one another to achieve the maximum number of arcs; this will be indicated by colouring in the next iteration. Remove the tags and boxes but retain the colouring.\\

\emph{Iteration 2:}

Box and colour the nucleotides of the new balanced set $\mathcal{B'}$, as shown below.

\begin{center}
\begin{tikzpicture}  [scale=0.625,auto=left]
\node at (2,2) {A};
\node at (3,2)[red] {G};
\node at (4,2)[draw,red] {A};
\node at (5,2)[red] {G};
\node at (6,2) {A};
\node at (7,2) {A};
\node at (8,2)[red] {C};
\node at (9,2)[red] {C};
\node at (10,2) {A};
\node at (11,2) {A};
\node at (12,2)[red] {G};
\node at (13,2)[draw,red] {U};
\node at (14,2)[draw,red] {U};
\node at (15,2)[red] {G};
\node at (16,2) {A};
\node at (17,2) {A};
\node at (18,2)[red] {C}; 
\node at (19,2)[red] {C};
\node at (20,2) {A};
\node at (21,2) {A};
\node at (22,2)[red] {G};
\node at (23,2)[draw,red] {A};
\node at (24,2)[red] {C};
\node at (3,0)[red] {C};
\node at (4,0)[draw,red] {A};
\node at (5,0)[red] {G};
\node at (6,0) {A};
\node at (7,0) {A};
\node at (8,0)[red] {C};
\node at (9,0)[red] {C};
\node at (10,0) {A};
\node at (11,0) {A};
\node at (12,0)[red] {G};
\node at (13,0)[draw,red] {U};
\node at (14,0)[draw,red] {U};
\node at (15,0)[red] {G};
\node at (16,0) {A};
\node at (17,0) {A}; 
\node at (18,0)[red] {C};
\node at (19,0)[red] {C};
\node at (20,0) {A};
\node at (21,0) {A};
\node at (22,0)[red] {G};
\node at (23,0)[draw,red] {A};
\node at (24,0)[red] {G};
\node at (25,0) {U};
\end{tikzpicture}
\end{center}

To avoid inducing an arc crossing, every arc in $M$ joining two nucleotides not in $\mathcal{B'}$ must enclose an equal number of the two nucleotide types in $\mathcal{B'}$ (boxed As and boxed Us). Tag the uncoloured nucleotides with the running difference with respect to $\mathcal{B'}$.

\begin{center}
\begin{tikzpicture}  [scale=0.625,auto=left]
\node at (2,2) {A};
\node at (3,2)[red] {G};
\node at (4,2)[draw,red] {A};
\node at (5,2)[red] {G};
\node at (6,2) {A};
\node at (7,2) {A};
\node at (8,2)[red] {C};
\node at (9,2)[red] {C};
\node at (10,2) {A};
\node at (11,2) {A};
\node at (12,2)[red] {G};
\node at (13,2)[draw,red] {U};
\node at (14,2)[draw,red] {U};
\node at (15,2)[red] {G};
\node at (16,2) {A};
\node at (17,2) {A};
\node at (18,2)[red] {C}; 
\node at (19,2)[red] {C};
\node at (20,2) {A};
\node at (21,2) {A};
\node at (22,2)[red] {G};
\node at (23,2)[draw,red] {A};
\node at (24,2)[red] {C};
\node at (3,0)[red] {C};
\node at (4,0)[draw,red] {A};
\node at (5,0)[red] {G};
\node at (6,0) {A};
\node at (7,0) {A};
\node at (8,0)[red] {C};
\node at (9,0)[red] {C};
\node at (10,0) {A};
\node at (11,0) {A};
\node at (12,0)[red] {G};
\node at (13,0)[draw,red] {U};
\node at (14,0)[draw,red] {U};
\node at (15,0)[red] {G};
\node at (16,0) {A};
\node at (17,0) {A}; 
\node at (18,0)[red] {C};
\node at (19,0)[red] {C};
\node at (20,0) {A};
\node at (21,0) {A};
\node at (22,0)[red] {G};
\node at (23,0)[draw,red] {A};
\node at (24,0)[red] {G};
\node at (25,0) {U};
\node at (2,3) {0};
\node at (3,3)[red] {};
\node at (4,3) {};
\node at (5,3)[red] {};
\node at (6,3) {1};
\node at (7,3) {1};
\node at (8,3)[red] {};
\node at (9,3)[red] {};
\node at (10,3) {1};
\node at (11,3) {1};
\node at (12,3)[red] {};
\node at (13,3) {};
\node at (14,3) {};
\node at (15,3)[red] {};
\node at (16,3) {-1};
\node at (17,3) {-1};
\node at (18,3)[red] {}; 
\node at (19,3)[red] {};
\node at (20,3) {-1};
\node at (21,3) {-1};
\node at (22,3)[red] {};
\node at (23,3) {};
\node at (24,3)[red] {};
\node at (3,1)[red] {};
\node at (4,1) {};
\node at (5,1)[red] {};
\node at (6,1) {1};
\node at (7,1) {1};
\node at (8,1)[red] {};
\node at (9,1)[red] {};
\node at (10,1) {1};
\node at (11,1) {1};
\node at (12,1)[red] {};
\node at (13,1) {};
\node at (14,1) {};
\node at (15,1)[red] {};
\node at (16,1) {-1};
\node at (17,1) {-1}; 
\node at (18,1)[red] {};
\node at (19,1)[red] {};
\node at (20,1) {-1};
\node at (21,1) {-1};
\node at (22,1)[red] {};
\node at (23,1) {};
\node at (24,1)[red] {};
\node at (25,1) {0};
\end{tikzpicture}
\end{center}

Uncoloured nucleotides whose tags have the same parity can pair only with one another. The corresponding version of $\tau$ (again shown only as a visual aid) is:

\begin{center}
\begin{tikzpicture}
[scale=0.475,auto=left]
\node[circle, scale=0.7, draw](n21) at (1.5,4)  {A};
\node[circle, scale=0.7, draw](n22) at (3.5,4)  {A};
\node[circle, scale=0.7, draw](n23) at (5.5,4)  {A};
\node[circle, scale=0.7, draw](n24) at (7.5,4)  {A};
\node[circle, scale=0.7, draw](n25) at (9.5,4)  {A};
\node[circle, scale=0.7, draw](n26) at (11.5,4)  {A};
\node[circle, scale=0.7, draw](n27) at (13.5,4)  {A};
\node[circle, scale=0.7, draw](n28) at (15.5,4)  {A};
\node[circle, scale=0.7, draw](n29) at (17.5,4)  {A};
\node[circle, scale=0.7, draw](n210) at (19.5,4)  {A};
\node[circle, scale=0.7, draw](n211) at (21.5,4)  {A};
\node[circle, scale=0.7, draw](n212) at (23.5,4)  {A};
\node[circle, scale=0.7, draw](n213) at (25.5,4)  {A};
\node[circle, scale=0.7, draw](n214) at (27.5,4)  {A};
\node[circle, scale=0.7, draw](n215) at (29.5,4)  {A};
\node[circle, scale=0.7, draw](n216) at (31.5,4)  {A};
\node[] at (1.5,5)  {1};
\node[] at (3.5,5)  {1};
\node[] at (5.5,5)  {1};
\node[] at (7.5,5)  {1};
\node[] at (9.5,5)  {-1};
\node[] at (11.5,5)  {-1};
\node[] at (13.5,5)  {-1};
\node[] at (15.5,5)  {-1};
\node[] at (17.5,5)  {1};
\node[] at (19.5,5)  {1};
\node[] at (21.5,5)  {1};
\node[] at (23.5,5)  {1};
\node[] at (25.5,5)  {-1};
\node[] at (27.5,5)  {-1};
\node[] at (29.5,5)  {-1};
\node[] at (31.5,5)  {-1};

\node[circle, scale=0.7, draw, fill=red!30](n31) at (2.5,7)  {GC};
\node[circle, scale=0.7, draw, fill=red!30](n32) at (6.5,7)  {CG};
\node[circle, scale=0.7, draw, fill=red!30](n33) at (10.5,7)  {GC};
\node[circle, scale=0.7, draw, fill=red!30](n34) at (14.5,7)  {CG};
\node[circle, scale=0.7, draw, fill=red!30](n35) at (18.5,7)  {GC};
\node[circle, scale=0.7, draw, fill=red!30](n36) at (22.5,7)  {CG};
\node[circle, scale=0.7, draw, fill=red!30](n37) at (26.5,7)  {GC};
\node[circle, scale=0.7, draw, fill=red!30](n38) at (30.5,7)  {CG};
\node[scale=0.7, draw, fill=red!30](n41) at (4.5,9)  {AU};
\node[scale=0.7, draw, fill=red!30](n42) at (12.5,9)  {UA};
\node[scale=0.7, draw, fill=red!30](n43) at (20.5,9)  {AU};
\node[scale=0.7, draw, fill=red!30](n44) at (28.5,9)  {UA};
\node[circle, scale=0.7, draw, fill=red!30](n51) at (8.5,11)  {GC};
\node[circle, scale=0.7, draw, fill=red!30](n52) at (24.5,11)  {CG};
\node[circle, scale=0.7, draw](n61) at (16.5,13)  {AU};
\node[] at (16.5,14)  {0, 0};

\foreach \from/\to in {n61/n51, n61/n52, n51/n41, n51/n42, n52/n43, n52/n44, n41/n31, n41/n32, n42/n33, n42/n34, n43/n35, n43/n36, n44/n37, n44/n38, n31/n21, n31/n22, n32/n23, n32/n24, n33/n25, n33/n26, n34/n27, n34/n28, n35/n29, n35/n210, n36/n211, n36/n212, n37/n213, n37/n214, n38/n215, n38/n216}
\draw (\from) -- (\to);
\end{tikzpicture}
\end{center}

All uncoloured nucleotides whose tags are even (this being the opposite parity to that of the nucleotides at depth 4) can pair only with one another. There are only two such nucleotides, namely the initial A and final U of the sequence $S$. They form a trivial new balanced set $\mathcal{B''}$ and so must pair with one another. Remove the tags and boxes but retain the colouring.\\

\emph{Iteration 3:}

Box and colour the nucleotides of the new balanced set $\mathcal{B''}$. 

\begin{center}
\begin{tikzpicture}  [scale=0.625,auto=left]
\node at (2,2)[draw,red] {A};
\node at (3,2)[red] {G};
\node at (4,2)[red] {A};
\node at (5,2)[red] {G};
\node at (6,2) {A};
\node at (7,2) {A};
\node at (8,2)[red] {C};
\node at (9,2)[red] {C};
\node at (10,2) {A};
\node at (11,2) {A};
\node at (12,2)[red] {G};
\node at (13,2)[red] {U};
\node at (14,2)[red] {U};
\node at (15,2)[red] {G};
\node at (16,2) {A};
\node at (17,2) {A};
\node at (18,2)[red] {C}; 
\node at (19,2)[red] {C};
\node at (20,2) {A};
\node at (21,2) {A};
\node at (22,2)[red] {G};
\node at (23,2)[red] {A};
\node at (24,2)[red] {C};
\node at (3,0)[red] {C};
\node at (4,0)[red] {A};
\node at (5,0)[red] {G};
\node at (6,0) {A};
\node at (7,0) {A};
\node at (8,0)[red] {C};
\node at (9,0)[red] {C};
\node at (10,0) {A};
\node at (11,0) {A};
\node at (12,0)[red] {G};
\node at (13,0)[red] {U};
\node at (14,0)[red] {U};
\node at (15,0)[red] {G};
\node at (16,0) {A};
\node at (17,0) {A}; 
\node at (18,0)[red] {C};
\node at (19,0)[red] {C};
\node at (20,0) {A};
\node at (21,0) {A};
\node at (22,0)[red] {G};
\node at (23,0)[red] {A};
\node at (24,0)[red] {G};
\node at (25,0)[draw,red] {U};
\end{tikzpicture}
\end{center}

At this point, only the nucleotides at depth 4 are uncoloured, which is the stopping criterion. Since all coloured nucleotides must pair with one another, and all nucleotides at depth 4 (the uncoloured nucleotides) are As, the nucleotides at depth 4 must remain unpaired in $M$.

We will use this fact to determine $M$ completely. Before doing so, we observe several properties that allow us to simplify the argument presented so far.
Consider the balanced sets $\mathcal{B}$, $\mathcal{B'}$, $\mathcal{B''}$ in iterations 1 to 3. Observe that the nucleotides in $S$ occurring at a given depth in $\tau$ are either all in, or all not in, the current balanced set. Furthermore the tags (with respect to this balanced set) of all nucleotides at a given depth share the same parity, and the parity switches between successive tagged depths. As a result, the labelled tree $\tau$ for each iteration 1 to 3 can be condensed as follows:

\begin{figure}[H]
\centering
\subfloat[Iteration 1]{
\begin{minipage}{0.32\textwidth}
\begin{tikzpicture}[scale=0.6,auto=left]
\node[circle, scale=0.7, draw, thick] (n01) at (1,1) {A} ;
\node[] at (-1.5,1) {Depth 4 $\rightarrow$};
\node[right of=n01] {0};
\node[scale=0.7,fill=red!30,  draw](n11) at (1,3)  {GC};
\node[circle, scale=0.7, draw](n21) at (1,5)  {AU};
\node[right of=n21] {1};
\node[scale=0.7,fill=red!30,  draw](n31) at (1,7)  {GC};
\node[circle, scale=0.7, draw](n41) at (1,9)  {AU};
\node[right of=n41] {0};

\foreach \from/\to in {n01/n11, n11/n21, n21/n31, n31/n41}
\draw (\from) -- (\to);

\node[scale=1.5](arrow) at (5,7) {$\rightarrow$};
\end{tikzpicture}
\end{minipage}\hfill
}
\subfloat[Iteration 2]{
\begin{minipage}{0.32\textwidth}
\hspace{2cm}
\begin{tikzpicture}[scale=0.6,auto=left]
\node[circle, scale=0.7, draw, thick] (n01) at (8,1) {A} ;
\node[right of=n01] {1};
\node[circle, scale=0.7,fill=red!30, draw](n11) at (8,3)  {GC};
\node[scale=0.7,fill=red!30,  draw](n21) at (8,5)  {AU};
\node[circle, scale=0.7,fill=red!30, draw](n31) at (8,7)  {GC};
\node[circle, scale=0.7, draw](n41) at (8,9)  {AU};
\node[right of=n41] {0};

\foreach \from/\to in {n01/n11, n11/n21, n21/n31, n31/n41}
\draw (\from) -- (\to);

\node[scale=1.5](arrow) at (12.5,7) {$\rightarrow$};
\end{tikzpicture}
\end{minipage}\hfill
}
\subfloat[Iteration 3]{
\begin{minipage}{0.32\textwidth}
\hspace{2cm}
\begin{tikzpicture}[scale=0.6,auto=left]
\node[circle, scale=0.7, draw, thick] (n01) at (15,1) {A} ;
\node[circle, scale=0.7,fill=red!30,  draw](n11) at (15,3)  {GC};
\node[circle, scale=0.7,fill=red!30,  draw](n21) at (15,5)  {AU};
\node[circle, scale=0.7,fill=red!30, draw](n31) at (15,7)  {GC};
\node[scale=0.7,fill=red!30,  draw](n41) at (15,9)  {AU};

\foreach \from/\to in {n01/n11, n11/n21, n21/n31, n31/n41}
\draw (\from) -- (\to);
\end{tikzpicture}
\end{minipage}
}
\caption{Simplified representation of iterations 1 to 3 using condensed tree}
\label{condensed}
\end{figure}

In the condensed tree, all vertices at a given depth in the original tree $\tau$ are represented by a single vertex labelled with either AU or GC according to the nucleotide types appearing at that depth in $\tau$. 
The numbers placed at uncoloured depths give the parity of the running difference with respect to the current balanced set.

We now show that $M$ must equal the arc set $R$ corresponding to~$T$.
Since all nucleotides in $S$ occurring at depth 4 in $\tau$ remain unpaired in $M$, we may remove them from $\tau$ to obtain the labelled height 3 saturated perfect binary tree $\widehat{\tau}$, namely:

\begin{center}
\begin{tikzpicture}
[scale=0.5,auto=left]
\node[circle, scale=0.7, draw](n31) at (2.5,7)  {GC};
\node[circle, scale=0.7, draw](n32) at (6.5,7)  {CG};
\node[circle, scale=0.7, draw](n33) at (10.5,7)  {GC};
\node[circle, scale=0.7, draw](n34) at (14.5,7)  {CG};
\node[circle, scale=0.7, draw](n35) at (18.5,7)  {GC};
\node[circle, scale=0.7, draw](n36) at (22.5,7)  {CG};
\node[circle, scale=0.7, draw](n37) at (26.5,7)  {GC};
\node[circle, scale=0.7, draw](n38) at (30.5,7)  {CG};
\node[circle, scale=0.7, draw](n41) at (4.5,9)  {AU};
\node[circle, scale=0.7, draw](n42) at (12.5,9)  {UA};
\node[circle, scale=0.7, draw](n43) at (20.5,9)  {AU};
\node[circle, scale=0.7, draw](n44) at (28.5,9)  {UA};
\node[circle, scale=0.7, draw](n51) at (8.5,11)  {GC};
\node[circle, scale=0.7, draw](n52) at (24.5,11)  {CG};
\node[circle, scale=0.7, draw](n61) at (16.5,13)  {AU};

\foreach \from/\to in {n61/n51, n61/n52, n51/n41, n51/n42, n52/n43, n52/n44, n41/n31, n41/n32, n42/n33, n42/n34, n43/n35, n43/n36, n44/n37, n44/n38}
\draw (\from) -- (\to);
\end{tikzpicture}
\end{center}

The corresponding nucleotide sequence $\widehat{S}$ is:

\begin{center}
\begin{tikzpicture}
\node at (1,1) {A};
\node at (1.5,1) {G};
\node at (2,1) {A};
\node at (2.5,1) {G};
\node at (3,1) {C};
\node at (3.5,1) {C};
\node at (4,1) {G};
\node at (4.5,1) {U};
\node at (5,1) {U};
\node at (5.5,1) {G};
\node at (6,1) {C};
\node at (6.5,1) {C};
\node at (7,1) {G};
\node at (7.5,1) {A};
\node at (8,1) {C};
\node at (8.5,1) {C};
\node at (9,1) {A};
\node at (9.5,1) {G};
\node at (10,1) {C};
\node at (10.5,1) {C};
\node at (11,1) {G};
\node at (11.5,1) {U};
\node at (12,1) {U};
\node at (12.5,1) {G};
\node at (13,1) {C};
\node at (13.5,1) {C};
\node at (14,1) {G};
\node at (14.5,1) {A};
\node at (15,1) {G};
\node at (15.5,1) {U};
\end{tikzpicture}
\end{center}

We claim that all nucleotides in $\widehat{S}$ occurring at a given depth in $\widehat{\tau}$ must pair with one another in~$M$. By then considering arcs joining nucleotides in $\widehat{S}$ that occur at depth 0 to 3 in $\widehat{\tau}$ in that order, and applying the condition of no arc crossings, we obtain the following nested sequence of arcs:
\vspace{-2em}
\begin{center}
\begin{tikzpicture}[scale=0.73]
\node at (1,1) {A};
\node at (1.5,1) {G};
\node at (2,1) {A};
\node at (2.5,1) {G};
\node at (3,1) {C};
\node at (3.5,1) {C};
\node at (4,1) {G};
\node at (4.5,1) {U};
\node at (5,1) {U};
\node at (5.5,1) {G};
\node at (6,1) {C};
\node at (6.5,1) {C};
\node at (7,1) {G};
\node at (7.5,1) {A};
\node at (8,1) {C};
\node at (8.5,1) {C};
\node at (9,1) {A};
\node at (9.5,1) {G};
\node at (10,1) {C};
\node at (10.5,1) {C};
\node at (11,1) {G};
\node at (11.5,1) {U};
\node at (12,1) {U};
\node at (12.5,1) {G};
\node at (13,1) {C};
\node at (13.5,1) {C};
\node at (14,1) {G};
\node at (14.5,1) {A};
\node at (15,1) {G};
\node at (15.5,1) {U};
\draw[bend left=65] (1,1.25) to (15.5,1.25);
\draw[bend left=65] (1.5,1.25) to (8,1.25);
\draw[bend left=65] (8.5,1.25) to (15,1.25);
\draw[bend left=65] (2,1.25) to (4.5,1.25);
\draw[bend left=65] (5,1.25) to (7.5,1.25);
\draw[bend left=65] (9,1.25) to (11.5,1.25);
\draw[bend left=65] (12,1.25) to (14.5,1.25);
\draw[bend left=65] (2.5,1.25) to (3,1.25);
\draw[bend left=65] (3.5,1.25) to (4,1.25);
\draw[bend left=65] (9.5,1.25) to (10,1.25);
\draw[bend left=65] (10.5,1.25) to (11,1.25);
\draw[bend left=65] (5.5,1.25) to (6,1.25);
\draw[bend left=65] (6.5,1.25) to (7,1.25);
\draw[bend left=65] (12.5,1.25) to (13,1.25);
\draw[bend left=65] (13.5,1.25) to (14,1.25);
\end{tikzpicture}
\end{center}

This arc set corresponds exactly to the arc set of $M$ (because the nucleotides occurring at depth 4 in $\tau$ remain unpaired in~$M$). Therefore $M$ equals $R$, as required.

It remains to prove the claim. To show that all nucleotides in $\widehat{S}$ occurring at depth $j$ ($0\leq j \leq 3$) in $\widehat{\tau}$ must pair with one another in $M$, we modify the previous process involving boxes, colouring, and tags so that the last remaining uncoloured depth is $j$ (rather than 4, as in Figure~\ref{condensed}). We illustrate this modified process for the case $j=1$, using a condensed tree representation for $\widehat{\tau}$. The argument for other values of $j$ is similar.

\begin{figure}[H]
\centering
\subfloat[Iteration 1]{
\begin{minipage}{0.49\textwidth}
\hspace{1.2cm}
\begin{tikzpicture}[scale=0.6,auto=left]
\node[circle, scale=0.7,  draw](n11) at (1,3)  {GC};
\node[right of=n11] {0};
\node[scale=0.7,fill=red!30, draw](n21) at (1,5)  {AU};
\node[circle, scale=0.7, draw, thick](n31) at (1,7)  {GC};
\node[] at (-1.5,7) {Depth 1 $\rightarrow$};
\node[right of=n31] {1};
\node[scale=0.7,fill=red!30, draw](n41) at (1,9)  {AU};

\foreach \from/\to in { n11/n21, n21/n31, n31/n41}
\draw (\from) -- (\to);

\node[scale=1.5](arrow) at (7.5,7) {$\rightarrow$};
\end{tikzpicture}
\end{minipage}
}
\subfloat[Iteration 2]{
\begin{minipage}{0.49\textwidth}
\hspace{3.2cm}
\begin{tikzpicture}[scale=0.6,auto=left]
\node[scale=0.7,fill=red!30, draw](n11) at (8,3)  {GC};
\node[circle, scale=0.7,fill=red!30, draw](n21) at (8,5)  {AU};

\node[circle, scale=0.7, draw, thick](n31) at (8,7)  {GC};
\node[right of=n31] {0};
\node[circle, scale=0.7,fill=red!30, draw](n41) at (8,9)  {AU};
\node[right of=n41] {};

\foreach \from/\to in {n11/n21, n21/n31, n31/n41}
\draw (\from) -- (\to);
\end{tikzpicture}
\end{minipage}
}
\caption{Sequence of condensed trees for the case $j=1$}
\label{condensed-j1}
\end{figure}

In iteration 1 of Figure~\ref{condensed-j1}, box and colour the even depths (having opposite parity to $j=1$). The nucleotides occurring at these depths belong to a balanced set $\mathcal{B}$, and so they must all pair with one another. Place at uncoloured depths the parity of the running difference with respect to $\mathcal{B}$.
Nucleotides whose running difference has even parity (opposite to that of the nucleotides at depth $j=1$) form a new balanced set $\mathcal{B'}$.

In iteration 2 of Figure~\ref{condensed-j1}, box and colour the depth at which the nucleotides of the new balanced set $\mathcal{B'}$ occur. All nucleotides at coloured depths must pair with one another. The only depth remaining uncoloured is $j=1$, all of whose nucleotides can therefore pair only with one another. The existence of $\widehat{\tau}$ (in which every nucleotide is paired) then shows that the nucleotides at depth $j=1$ must all pair with one another. This establishes the claim.

We shall prove Theorem \ref{main} in Section~\ref{sec:proof} by generalizing this example.

\section{Proof of main result (Theorem~\ref{main})}
\label{sec:proof}
In this section we prove Theorem \ref{main} according to the model described in Section~\ref{sec:pics}. 

We shall require two preliminary results. 
Observation~\ref{importantobs} states that we cannot ``jump over'' a depth of a labelled tree as it is traversed to produce its corresponding nucleotide sequence. 
Lemma~\ref{importantlemma} counts the number of nucleotides at a fixed depth that occur in a subsequence of the nucleotide sequence corresponding to the natural labelling of a P-unsaturated perfect floral tree. 
The four cases of Lemma~\ref{importantlemma} are illustrated in Figure~\ref{fourcases} using a tree of height $3$, with nucleotides $N_1$ and $N_2$ marked in red in the tree representation (left) and corresponding arc representation (right), and the nucleotides occurring at depth $h$ between $N_1$ and $N_2$ marked in green.

\begin{obs}\label{importantobs}
Let $S$ be the nucleotide sequence corresponding to a labelled tree. Then adjacent nucleotides in $S$ occur either at the same depth or at depths that differ by one.
\end{obs}

\begin{lem}\label{importantlemma}
Let $S$ be the nucleotide sequence corresponding to the natural labelling of a P-unsaturated perfect floral tree of height $n+1$. Let $N_1$ be a nucleotide at depth $k$. Suppose at least one nucleotide following $N_1$ in $S$ is at depth $\ell$, and let $N_2$ be the first such nucleotide.

Further suppose the subsequence of $S$ starting at $N_1$ and ending at $N_2$ passes through depth $h\not\in\left\{{k, \ell}\right\}$ at least once. Then the total number of nucleotides at depth $h$ occurring between $N_1$ and $N_2$ in $S$ is
\[
  \begin{cases}
                                   2 & \text{if $\ell = k$ and $h<k$} \\
                                   2^{h-k+1} & \text{if $\ell = k$ and $k<h<n+1$} \\
                                   1 & \text{if $k<h<\ell$} \\
                                   odd & \text{if $\ell<h<k$}.
  \end{cases}
\]
\end{lem}

\begin{proof}
Consider the path through the tree corresponding to the subsequence of $S$ that starts at $N_1$ and ends at $N_2$.

\begin{description}
\item[Case 1: $\ell = k$ and $h<k$.] 
The path passes through depth $h$ exactly once in the upward direction and exactly once in the downward direction. 

\item[Case 2: $\ell = k$ and $k<h<n+1$.]
The nucleotides $N_1$ and $N_2$ form a base pair (by definition of $N_2$), and the path traverses the entire subtree rooted at the vertex labelled by the base pair~$N_1N_2$.

\item[Case 3: $k<h<\ell$.]
By definition of $N_2$, the path passes through depth~$h$ exactly once. 

\item[Case 4: $\ell<h<k$.]
Consider the subtrees rooted at the vertices at depth~$h$. Nucleotide $N_1$ occurs in one of these subtrees, say $\sigma$, and in $S$ it lies between the two nucleotides labelling the root of~$\sigma$. To reach $N_2$ from $N_1$, the path passes through the second of these two nucleotides and then through both the nucleotides labelling the root of all the subtrees (if any) lying to the right of $\sigma$ in the tree before reaching depth~$\ell$.
\end{description}
\end{proof}

\begin{figure}[H]
\begin{center}
\begin{tikzpicture}[scale=0.5]
\node[scale=0.7, draw, circle](n00022) at (6.5,5)  {C};
\node[scale=0.7, draw, circle](n00023) at (5,5)  {C};
\node[scale=0.7, draw, circle](n00024) at (8,5)  {C};
\node[scale=0.7, draw, circle](n31) at (2.5,7)  {AU};
\node[scale=0.7, draw, circle](n32) at (6.5,7)  {U\textcolor{red}{\Large \bf A}}; 
\node[scale=0.7, draw, circle](n33) at (10.5,7)  {\textcolor{red}{\Large \bf A}U};
\node[scale=0.7, draw, circle](n34) at (14.5,7)  {UA};
\node[scale=0.7, draw, circle](n41) at (4.5,9)  {G\textcolor{black!20!green}{\Large \bf C}};
\node[scale=0.7, draw, circle](n42) at (12.5,9)  {\textcolor{black!20!green}{\Large \bf C}G};
\node[scale=0.7, draw, circle](n51) at (8.5,11)  {AU};

\foreach \from/\to in {n51/n41, n51/n42, n41/n31, n41/n32, n42/n33, n42/n34, n32/n00024, n32/n00023, n32/n00022}
\draw (\from) -- (\to);
\end{tikzpicture}
\hspace{0.4cm}
\begin{tikzpicture}[scale=0.9]
\node at (1,1) {A};
\node at (1.5,1) {G};
\node at (2,1) {A};
\node at (2.5,1) {U};
\node at (3,1) {U};
\node at (3.5,1) {C};
\node at (4,1) {C};
\node at (4.5,1) {C};
\node[black!20!red, scale=1.2] at (5,1) {A};
\node[black!20!green, scale=1.2] at (5.5,1) {C};
\node[black!20!green, scale=1.2] at (6,1) {C};
\node[black!20!red, scale=1.2] at (6.5,1) {A};
\node at (7,1) {U};
\node at (7.5,1) {U};
\node at (8,1) {A};
\node at (8.5,1) {G};
\node at (9,1) {U};

\node[scale=0.8] at (5,0.5) {$N_1$};
\node[scale=0.8] at (6.5,0.5) {$N_2$};
\draw[bend left=65] (1,1.25) to (9,1.25);
\draw[bend left=65] (1.5,1.25) to (5.5,1.25);
\draw[bend left=65] (6,1.25) to (8.5,1.25);
\draw[bend left=65] (2,1.25) to (2.5,1.25);
\draw[bend left=65] (3,1.25) to (5,1.25);
\draw[bend left=65] (6.5,1.25) to (7,1.25);
\draw[bend left=65] (7.5,1.25) to (8,1.25);
\end{tikzpicture}
\end{center}
{\bf Case 1: $\ell = k$ and $h<k$} (illustrated by $\ell = k = 2$ and $h = 1$).

\begin{center}
\begin{tikzpicture}[scale=0.5]
\node[scale=0.7, draw, circle](n00022) at (6.5,5)  {C};
\node[scale=0.7, draw, circle](n00023) at (5,5)  {C};
\node[scale=0.7, draw, circle](n00024) at (8,5)  {C};
\node[scale=0.7, draw, circle](n31) at (2.5,7)  {\textcolor{black!20!green}{\Large \bf AU}};
\node[scale=0.7, draw, circle](n32) at (6.5,7)  {\textcolor{black!20!green}{\Large \bf UA}};
\node[scale=0.7, draw, circle](n33) at (10.5,7)  {AU};
\node[scale=0.7, draw, circle](n34) at (14.5,7)  {UA};
\node[scale=0.7, draw, circle](n41) at (4.5,9)  {\textcolor{red}{\Large \bf GC}};
\node[scale=0.7, draw, circle](n42) at (12.5,9)  {CG};
\node[scale=0.7, draw, circle](n51) at (8.5,11)  {AU};

\foreach \from/\to in {n51/n41, n51/n42, n41/n31, n41/n32, n42/n33, n42/n34, n32/n00024, n32/n00023, n32/n00022}
\draw (\from) -- (\to);
\end{tikzpicture}
\hspace{0.4cm}
\begin{tikzpicture}[scale=0.9]
\node at (1,1) {A};
\node[black!20!red, scale=1.2] at (1.5,1) {G};
\node[black!20!green, scale=1.2] at (2,1) {A};
\node[black!20!green, scale=1.2] at (2.5,1) {U};
\node[black!20!green, scale=1.2] at (3,1) {U};
\node at (3.5,1) {C};
\node at (4,1) {C};
\node at (4.5,1) {C};
\node[black!20!green, scale=1.2] at (5,1) {A};
\node[black!20!red, scale=1.2] at (5.5,1) {C};
\node at (6,1) {C};
\node at (6.5,1) {A};
\node at (7,1) {U};
\node at (7.5,1) {U};
\node at (8,1) {A};
\node at (8.5,1) {G};
\node at (9,1) {U};

\node[scale=0.8] at (1.5,0.5) {$N_1$};
\node[scale=0.8] at (5.5,0.5) {$N_2$};
\draw[bend left=65] (1,1.25) to (9,1.25);
\draw[bend left=65] (1.5,1.25) to (5.5,1.25);
\draw[bend left=65] (6,1.25) to (8.5,1.25);
\draw[bend left=65] (2,1.25) to (2.5,1.25);
\draw[bend left=65] (3,1.25) to (5,1.25);
\draw[bend left=65] (6.5,1.25) to (7,1.25);
\draw[bend left=65] (7.5,1.25) to (8,1.25);
\end{tikzpicture}
\end{center}
{\bf Case 2: $\ell = k$ and $k<h<n+1$} (illustrated by $\ell = k = 1$ and $h = 2$).

\begin{center}
\begin{tikzpicture}[scale=0.5]
\node[scale=0.7, draw, circle](n00022) at (6.5,5)  {C};
\node[scale=0.7, draw, circle](n00023) at (5,5)  {C};
\node[scale=0.7, draw, circle](n00024) at (8,5)  {C};
\node[scale=0.7, draw, circle](n31) at (2.5,7)  {\textcolor{red}{\Large \bf A}U};
\node[scale=0.7, draw, circle](n32) at (6.5,7)  {UA};
\node[scale=0.7, draw, circle](n33) at (10.5,7)  {AU};
\node[scale=0.7, draw, circle](n34) at (14.5,7)  {UA};
\node[scale=0.7, draw, circle](n41) at (4.5,9)  {\textcolor{black!20!green}{\Large \bf G}C};
\node[scale=0.7, draw, circle](n42) at (12.5,9)  {CG};
\node[scale=0.7, draw, circle](n51) at (8.5,11)  {\textcolor{red}{\Large \bf A}U};

\foreach \from/\to in {n51/n41, n51/n42, n41/n31, n41/n32, n42/n33, n42/n34, n32/n00024, n32/n00023, n32/n00022}
\draw (\from) -- (\to);
\end{tikzpicture}
\hspace{0.4cm}
\begin{tikzpicture}[scale=0.9]
\node[black!20!red, scale=1.2] at (1,1) {A};
\node[black!20!green, scale=1.2] at (1.5,1) {G};
\node[black!20!red, scale=1.2] at (2,1) {A};
\node at (2.5,1) {U};
\node at (3,1) {U};
\node at (3.5,1) {C};
\node at (4,1) {C};
\node at (4.5,1) {C};
\node at (5,1) {A};
\node at (5.5,1) {C};
\node at (6,1) {C};
\node at (6.5,1) {A};
\node at (7,1) {U};
\node at (7.5,1) {U};
\node at (8,1) {A};
\node at (8.5,1) {G};
\node at (9,1) {U};
\node[scale=0.8] at (1,0.5) {$N_1$};
\node[scale=0.8] at (2,0.5) {$N_2$};
\draw[bend left=65] (1,1.25) to (9,1.25);
\draw[bend left=65] (1.5,1.25) to (5.5,1.25);
\draw[bend left=65] (6,1.25) to (8.5,1.25);
\draw[bend left=65] (2,1.25) to (2.5,1.25);
\draw[bend left=65] (3,1.25) to (5,1.25);
\draw[bend left=65] (6.5,1.25) to (7,1.25);
\draw[bend left=65] (7.5,1.25) to (8,1.25);
\end{tikzpicture}
\end{center}
{\bf Case 3:} $k<h<\ell$ (illustrated by $\ell = 2$ and $k = 0$ and $h = 1$).

\begin{center}
\begin{tikzpicture}[scale=0.5]
\node[scale=0.7, draw, circle](n00022) at (6.5,5)  {C};
\node[scale=0.7, draw, circle](n00023) at (5,5)  {C};
\node[scale=0.7, draw, circle](n00024) at (8,5)  {C};
\node[scale=0.7, draw, circle](n31) at (2.5,7)  {A\textcolor{red}{\Large \bf U}};
\node[scale=0.7, draw, circle](n32) at (6.5,7)  {UA};
\node[scale=0.7, draw, circle](n33) at (10.5,7)  {AU};
\node[scale=0.7, draw, circle](n34) at (14.5,7)  {UA};
\node[scale=0.7, draw, circle](n41) at (4.5,9)  {G\textcolor{black!20!green}{\Large \bf C}};
\node[scale=0.7, draw, circle](n42) at (12.5,9)  {\textcolor{black!20!green}{\Large \bf CG}};
\node[scale=0.7, draw, circle](n51) at (8.5,11)  {A\textcolor{red}{\Large \bf U}};

\foreach \from/\to in {n51/n41, n51/n42, n41/n31, n41/n32, n42/n33, n42/n34, n32/n00024, n32/n00023, n32/n00022}
\draw (\from) -- (\to);
\end{tikzpicture}
\hspace{0.4cm}
\begin{tikzpicture}[scale=0.9]
\node at (1,1) {A};
\node at (1.5,1) {G};
\node at (2,1) {A};
\node[black!20!red, scale=1.2] at (2.5,1) {U};
\node at (3,1) {U};
\node at (3.5,1) {C};
\node at (4,1) {C};
\node at (4.5,1) {C};
\node at (5,1) {A};
\node[black!20!green, scale=1.2] at (5.5,1) {C};
\node[black!20!green, scale=1.2] at (6,1) {C};
\node at (6.5,1) {A};
\node at (7,1) {U};
\node at (7.5,1) {U};
\node at (8,1) {A};
\node[black!20!green, scale=1.2] at (8.5,1) {G};
\node[black!20!red, scale=1.2] at (9,1) {U};

\node[scale=0.8] at (2.5,0.5) {$N_1$};
\node[scale=0.8] at (9,0.5) {$N_2$};
\draw[bend left=65] (1,1.25) to (9,1.25);
\draw[bend left=65] (1.5,1.25) to (5.5,1.25);
\draw[bend left=65] (6,1.25) to (8.5,1.25);
\draw[bend left=65] (2,1.25) to (2.5,1.25);
\draw[bend left=65] (3,1.25) to (5,1.25);
\draw[bend left=65] (6.5,1.25) to (7,1.25);
\draw[bend left=65] (7.5,1.25) to (8,1.25);
\end{tikzpicture}
\end{center}
{\bf Case 4: $\ell < h < k$} (illustrated by $\ell = 0$ and $k = 2$ and $h = 1$).
\caption{Illustration of the four cases of Lemma \ref{importantlemma}}
\label{fourcases}
\end{figure}

The crucial property which we shall use from Lemma \ref{importantlemma} is that the nucleotide count is even if $\ell = k$ and is odd if $\ell\not= k$. Note that we have excluded from Lemma \ref{importantlemma} the case where $\ell = k$ and $h=n+1$, because the count of the nucleotides at depth $n+1$ of the tree is then not determined by the conditions of the lemma. 

We now prove Theorem \ref{main}, subject to two Claims whose proof will be given directly after the main argument. Let $T$ be a P-unsaturated perfect floral tree of height $n+1\geq 1$, and let $S$ be the nucleotide sequence corresponding to the natural labelling $\tau$ of $T$. Suppose a maximum-size arc set $M$ is applied to the sequence $S$.\\

\noindent {\bf Claim 1:} All nucleotides in $S$ occurring at depth $n+1$ in $\tau$ remain unpaired in $M$.\\

By Claim 1, we may remove the nucleotides at depth $n+1$ in $\tau$ to obtain a labelled saturated perfect binary tree $\widehat{\tau}$ of height $n\geq 0$. Let $\widehat{S}$ be the sequence corresponding to $\widehat{\tau}$.\\

\noindent {\bf Claim 2:} For each $j$ satisfying $0\leq j \leq n $, all nucleotides in $\widehat{S}$ occurring at depth $j$ in $\widehat{\tau}$ must pair with one another in $M$.\\

Apply Claim 2 to successive values of $j$ starting from $j=0$. The case $j=0$ forces the leftmost and rightmost nucleotide of $\widehat{S}$ to be paired, giving the outermost arc shown below. Each successive case $j\geq 1$, together with the condition that there are no arc crossings, forces a further $2^j$ arcs of $M$ to take the nested form shown below.

\begin{figure}[H]
\centering
\begin{tikzpicture}
\node at (1,1) {A};
\node at (1.5,1) {G};
\node at (2,1) {A};
\node at (2.75,1) {$\dots$};
\node at (3,1) {};
\node at (3.5,1) {U};
\node at (4,1) {U};
\node at (4.75,1) {$\dots$};
\node at (5,1) {};
\node at (5.5,1) {A};
\node at (6,1) {C};
\node at (6.5,1) {C};
\node at (7,1) {A};
\node at (7.75,1) {$\dots$};
\node at (8,1) {};
\node at (8.5,1) {U};
\node at (9,1) {U};
\node at (9.75,1) {$\dots$};
\node at (10,1) {};
\node at (10.5,1) {A};
\node at (11,1) {G};
\node at (11.5,1) {U};
\node at (6.25, 4.3) {$j=0$};
\node at (3.75, 2.8) {$j=1$};
\node at (8.75, 2.8) {$j=1$};
\node at (2.75, 1.85) {$j=2$};
\node at (4.75, 1.85) {$j=2$};
\node at (7.75, 1.85) {$j=2$};
\node at (9.75, 1.85) {$j=2$};
\draw[bend left=65] (1,1.25) to (11.5,1.25);
\draw[bend left=65] (1.5,1.25) to (6,1.25);
\draw[bend left=65] (6.5,1.25) to (11,1.25);
\draw[bend left=65] (2,1.25) to (3.5,1.25);
\draw[bend left=65] (4,1.25) to (5.5,1.25);
\draw[bend left=65] (7,1.25) to (8.5,1.25);
\draw[bend left=65] (9,1.25) to (10.5,1.25);
\end{tikzpicture}
\caption{Forced secondary structure}
\end{figure}

Note that this repeated application of Claim 2 to $\widehat{S}$ establishes that $\widehat{S}$ admits the unique maximum-size arc set shown above, corresponding to a saturated perfect binary tree, and so $\widehat{S}$ is a design. Furthermore, by Claim 1 this arc set corresponds exactly to the arc set~$M$. Therefore $S$ is a design whose unique maximum-size arc set $M$ equals the secondary structure $R$ corresponding to~$T$. It remains to prove Claims 1 and 2.

\begin{proof}[Proof of Claim 1.]
\noindent We use the following algorithm to iteratively tag the nucleotides of the labelled tree $\tau$. This algorithm is modelled on the examples of Section~\ref{sec:pics} involving boxes, colouring, and tags. However, instead of indicating successive balanced sets by $\mathcal{B}, \mathcal{B'}, \mathcal{B''},\dots$ and using boxes to indicate the current balanced set, we update the current balanced set $\mathcal{B}$ by replacing it with a set $\mathcal{B'}$ that will be shown to be balanced.
\vspace{2em}

\noindent {\bf Algorithm 1.}
\begin{description}
\item[\quad Input:] Labelled tree $\tau$ and corresponding sequence $S$.
\item[\quad Step 1.] Initialize the set $\mathcal{B}$ to be all nucleotides in $S$ occurring in $\tau$ at depths whose parity is opposite to that of $n+1$.
\item[\quad Step 2.] Colour all nucleotides in $\mathcal{B}$. 
\item[\quad Step 3.] If only the nucleotides at depth $n+1$ are uncoloured, stop.
\item[\quad Step 4.] Tag all uncoloured nucleotides by the running difference with respect to $\mathcal{B}$.
\item[\quad Step 5.] Let $\mathcal{B}'$ be the set of all tagged nucleotides whose running difference has opposite parity to that of the leftmost nucleotide at depth $n+1$. Replace $\mathcal{B}$ by $\mathcal{B}'$.
\item[\quad Step 6.] Remove tags (but not the colouring) and go to Step 2.
\item[\quad Output:] Sequence $S$ with some nucleotides coloured.
\end{description}
\noindent
We shall establish that the following six properties hold for Algorithm 1.
\begin{prop}\label{prop1}
For each set $\mathcal{B}$, nucleotides at a given depth are either all in $\mathcal{B}$, or all not in~$\mathcal{B}$.
\end{prop}
\begin{prop}\label{prop2}
For each set $\mathcal{B}$, the nucleotides at depth $n+1$ are not coloured at Step 2 (and so, if Step 4 is reached, these nucleotides receive a tag and therefore Step 5 is well-defined).
\end{prop}
\begin{prop}\label{prop3}
For each set $\mathcal{B}$, the tags (with respect to $\mathcal{B}$) of all nucleotides at a given depth share the same parity.
\end{prop}
\begin{prop}\label{prop4}
Parity alternates between successive tagged depths.
\end{prop}
\begin{prop}\label{prop5}
Each set $\mathcal{B}$ is balanced.
\end{prop}
\begin{prop}\label{prop6}
Elements of each balanced set $\mathcal{B}$ must all pair with one another in $M$.
\end{prop}

Assume Properties \ref{prop1} to \ref{prop6} hold, and consider how Algorithm 1 terminates at Step 3. Step~2 colours all nucleotides in set $\mathcal{B}$, which by Property \ref{prop1} comprises all nucleotides at one or more depths. Therefore each application of Step 2 strictly decreases the number of uncoloured depths. This number is bounded below by 1 because, by Property \ref{prop2}, the nucleotides at depth $n+1$ are never coloured. Therefore eventually this number is reduced to 1 at Step 2, at which point only the nucleotides at depth $n+1$ are uncoloured and Algorithm 1 immediately terminates at Step 3. At this point, all nucleotides at other depths have been coloured (at some iteration of Step 2), and so belong to some balanced set by Property~\ref{prop5}; by Property \ref{prop6} these nucleotides must all pair with one another in $M$. The nucleotides at depth $n+1$ all have the same type (A or C) and so cannot pair with one another; therefore they remain unpaired in $M$, establishing Claim 1.

To complete the proof of Claim 1, we must prove Properties \ref{prop1} to \ref{prop6}.\\

We first show that Properties \ref{prop1} to \ref{prop4} hold for the initial set $\mathcal{B}$ (defined in Step 1).

Property \ref{prop1}: The initial set $\mathcal{B}$ is all nucleotides at depths whose parity is opposite to that of $n+1$.

Property \ref{prop2}: All nucleotides at depth $n+1$ lie outside $\mathcal{B}$ and so are uncoloured.

Properties \ref{prop3} and \ref{prop4}: From Step 4, the tagged nucleotides are exactly the uncoloured nucleotides and so occur at depths whose parity is the same as that of $n+1$. Let $N_1$ be an uncoloured nucleotide, and suppose another uncoloured nucleotide follows $N_1$ in $S$. Let $N_2$ be the first such uncoloured nucleotide. We may assume that at least one nucleotide from $\mathcal{B}$ occurs between $N_1$ and $N_2$ in $S$ (otherwise $N_1$ and $N_2$ occur at the same depth and receive identical tags with respect to $\mathcal{B}$).

Let $N_1$ occur at depth $k$ and $N_2$ at depth $\ell$. By Property \ref{prop1} for $\mathcal{B}$, all nucleotides at depth $\ell$ remain uncoloured at Step 2, and so $N_2$ is the first nucleotide at depth $\ell$ following $N_1$ in $S$. Observation \ref{importantobs} implies that $\ell\in\left\{{k-2, k, k+2}\right\}$, and all nucleotides in $\mathcal{B}$ occurring in the subsequence of $S$ starting at $N_1$ and ending at $N_2$ occur at a single depth $h\in\left\{{k-1, k+1}\right\}$.

The possibilities for $(\ell, h)$ are $(k, k-1), (k, k+1), (k+2, k+1),$ and $(k-2, k-1)$. By definition of $\mathcal{B}$ we have $h\not= n+1$ (which in fact implies that the case $(\ell, h)=(k, k+1)$ cannot occur). Then by Lemma \ref{importantlemma}, the number of nucleotides in $\mathcal{B}$ occurring in the subsequence of $S$ starting at $N_1$ and ending at $N_2$ is even if $\ell=k$ and is odd if $\ell\not=k$. Therefore the tags with respect to $\mathcal{B}$ of all uncoloured nucleotides at a given depth share the same parity, and the parity alternates between successive tagged depths.\\

We now assume Properties \ref{prop1} to \ref{prop4} hold for the current set $\mathcal{B}$, and show that they hold for the set $\mathcal{B'}$ defined in Step 5.

Property \ref{prop1}: By Properties~\ref{prop2} and \ref{prop3} for $\mathcal{B}$, all nucleotides at depth $n+1$ are tagged with respect to $\mathcal{B}$ and have the same parity. So $\mathcal{B'}$ comprises all tagged nucleotides whose running difference with respect to $\mathcal{B}$ has opposite parity to that of all the nucleotides at depth $n+1$. Then by Property \ref{prop3} for $\mathcal{B}$, we see that Property \ref{prop1} holds for $\mathcal{B'}$.

Property \ref{prop2}: nucleotides at depth $n+1$ do not belong to $\mathcal{B'}$ (from Step 5) and so are not coloured (in Step 2).

Properties \ref{prop3} and \ref{prop4}: We represent the application of Steps 2 and 4 to $\mathcal{B'}$ by the following sequence of ``condensed trees''.

\begin{figure}[H] 
\centering
\vspace{-0.3cm}
\hspace{-1cm}
\subfloat[\hspace{2.6cm}part~(a)]{
\begin{minipage}{0.14\textwidth} 
\begin{tikzpicture}[scale=0.41,auto=left]
\node[scale=1.5](arrow) at (8,9) {$\rightarrow$};
\node[circle, scale=0.7, draw, thick] (n01) at (1,1) {$\bullet$} ;
\node[right of=n01] {$p$};
\node[] at (-3, 1) {Depth $n+1$ $\rightarrow$};
\node[circle, scale=0.7,  draw](n11) at (1,3)  { };
\node[left of=n11] {};
\node[right of=n11] {$1-p$};
\node[circle, scale=0.7, draw](n21) at (1,5)  { };
\node[left of=n21] {};
\node[right of=n21] {$p$};
\node[circle, scale=0.7,  draw](n31) at (1,7)  { };
\node[left of=n31] {};
\node[right of=n31] {$1-p$};
\node[circle, scale=0.7, draw](n41) at (1,9)  { };
\node[left of=n41] {};
\node[right of=n41] {$p$};
\node[circle, scale=0.7,  draw](n51) at (1,11)  { };
\node[left of=n51] {};
\node[right of=n51] {$1-p$};
\node[circle, scale=0.7,  draw](n61) at (1,13)  { };
\node[left of=n61] {};
\node[right of=n61] {$p$};
\node[circle, scale=0.7, draw](n71) at (1,15)  { };
\node[left of=n71] {};
\node[right of=n71] {$1-p$};

\foreach \from/\to in {n01/n11, n11/n21, n21/n31, n31/n41, n41/n51, n51/n61, n61/n71}
\draw [dashed] (\from) -- (\to);
\draw [dashed] (n71) -- (1, 17);

\end{tikzpicture}
\end{minipage}
}
\hspace{4cm}
\subfloat[part~(b)]{
\begin{minipage}{0.14\textwidth} 
\begin{tikzpicture}[scale=0.41,auto=left]
\node[scale=1.5](arrow) at (7,9) {$\rightarrow$};
\node[circle, scale=0.7, draw, thick] (n01) at (1,1) {$\bullet$} ;
\node[left of=n01] {};
\node[right of=n01] {};
\node[circle, scale=0.7,fill=red!30, draw](n11) at (1,3)  { };
\node[left of=n11] {};
\node[right of=n11] {};
\node[circle, scale=0.7, draw](n21) at (1,5)  { };
\node[left of=n21] {};
\node[right of=n21] {};
\node[circle, scale=0.7,fill=red!30,  draw](n31) at (1,7)  { };
\node[left of=n31] {};
\node[right of=n31] {};
\node[circle, scale=0.7, draw](n41) at (1,9)  { };
\node[left of=n41] {};
\node[right of=n41] {};
\node[circle, scale=0.7,fill=red!30,  draw](n51) at (1,11)  { };
\node[left of=n51] {};
\node[right of=n51] {};
\node[circle, scale=0.7,  draw](n61) at (1,13)  { };
\node[left of=n11] {};
\node[right of=n11] {};
\node[circle, scale=0.7,fill=red!30, draw](n71) at (1,15)  { };
\node[left of=n21] {};
\node[right of=n21] {};

\foreach \from/\to in {n01/n11, n11/n21, n21/n31, n31/n41, n41/n51, n51/n61, n61/n71}
\draw [dashed] (\from) -- (\to);
\draw [dashed] (n71) -- (1, 17);

\end{tikzpicture}
\end{minipage}
}
\hspace{2cm}
\subfloat[part~(c)]{
\begin{minipage}{0.14\textwidth} 
\begin{tikzpicture}[scale=0.41,auto=left]
\node[circle, scale=0.7, draw, thick] (n01) at (1,1) {$\bullet$} ;
\node[left of=n01] {};
\node[right of=n01] {$p'$};
\node[circle, scale=0.7,fill=red!30, draw](n11) at (1,3)  { };
\node[left of=n11] {};
\node[right of=n11] {};
\node[circle, scale=0.7, draw](n21) at (1,5)  { };
\node[left of=n21] {};
\node[right of=n21] {$1-p'$};
\node[circle, scale=0.7,fill=red!30,  draw](n31) at (1,7)  { };
\node[left of=n31] {};
\node[right of=n31] {};
\node[circle, scale=0.7, draw](n41) at (1,9)  { };
\node[left of=n41] {};
\node[right of=n41] {$p'$};
\node[circle, scale=0.7,fill=red!30,  draw](n51) at (1,11)  { };
\node[left of=n51] {};
\node[right of=n51] {};
\node[circle, scale=0.7,  draw](n61) at (1,13)  { };
\node[left of=n61] {};
\node[right of=n61] {$1-p'$};
\node[circle, scale=0.7,fill=red!30, draw](n71) at (1,15)  { };
\node[left of=n71] {};
\node[right of=n71] {};

\foreach \from/\to in {n01/n11, n11/n21, n21/n31, n31/n41, n41/n51, n51/n61, n61/n71}
\draw [dashed] (\from) -- (\to);
\draw [dashed] (n71) -- (1, 17);

\end{tikzpicture}
\end{minipage}
}
\caption{Application of Steps 2 and 4 to new balanced set $\mathcal{B}'$}
\label{steps24}
\end{figure}

Figure~\ref{steps24}(a) represents the tree prior to applying Step 5 to produce $\mathcal{B'}$; only the depths containing nucleotides that are tagged with respect to $\mathcal{B}$ (and which are necessarily uncoloured) are shown. (Previously coloured nucleotides may occur at other depths not shown.) By Property~\ref{prop1} for~$\mathcal{B}$, all nucleotides at a given such depth may be represented by a single vertex. By Property~\ref{prop3} for~$\mathcal{B}$, each such vertex may be assigned a single parity ($p$ or $1-p$), representing the parity of the tag shared by all nucleotides at that depth. By Properties~\ref{prop2} and \ref{prop4} for~$\mathcal{B}$, the nucleotides at depth $n+1$ appear as a tagged vertex in the condensed tree, say with parity $p$, and the parity of vertices at successive tagged depths alternates. The set $\mathcal{B'}$ therefore comprises all nucleotides at depths tagged with parity $1-p$. These depths are coloured in Figure~\ref{steps24}(b) to indicate that the nucleotides at these depths are coloured when Step 2 is applied to~$\mathcal{B'}$.

Figure~\ref{steps24}(c) shows the parity of the tags assigned to nucleotides when Step 4 is applied to~$\mathcal{B'}$, as we now describe; this establishes Properties \ref{prop3} and \ref{prop4} for $\mathcal{B'}$. The uncoloured nucleotides are tagged in Step 4 by the running difference with respect to $\mathcal{B'}$. Let $N_1$ be an uncoloured nucleotide, and suppose another uncoloured nucleotide follows $N_1$ in $S$. Let $N_2$ be the first such uncoloured nucleotide. We may assume that at least one nucleotide from $\mathcal{B}'$ occurs between $N_1$ and $N_2$ in $S$ (otherwise $N_1$ and $N_2$ occur at the same depth and receive identical tags with respect to $\mathcal{B}'$).

Let $N_1$ occur at depth $k$ and $N_2$ at depth $\ell$. By Property \ref{prop1} for $\mathcal{B}'$, all nucleotides at depth~$\ell$ remain uncoloured at Step 2, and so $N_2$ is the first nucleotide at depth $\ell$ following $N_1$ in $S$. Observation \ref{importantobs} shows that $k$ and $\ell$ are either the same depth or are successive uncoloured depths. In both cases, all nucleotides in $\mathcal{B}'$ occurring in the subsequence of $S$ starting at $N_1$ and ending at~$N_2$ occur at a single depth $h$ (see Figure~\ref{steps24}(b)).

In order to apply Lemma \ref{importantlemma}, we verify the appropriate conditions on $k, h,$ and $\ell$. By Property~\ref{prop1} for $\mathcal{B}'$, all nucleotides at depth $h$ are in $\mathcal{B}'$ (coloured at Step 2 for $\mathcal{B}'$) and all nucleotides at depths $k$ and $\ell$ are not in $\mathcal{B}'$ (uncoloured). Therefore $h\not\in\left\{{k, \ell}\right\}$, and by Property \ref{prop2} for $\mathcal{B}'$ we have $h\not=n+1$. By Observation \ref{importantobs}, if $k\not=\ell$ then either $k<h<\ell$ or $\ell<h<k$. 

Then by Lemma \ref{importantlemma}, the number of nucleotides in $\mathcal{B}'$ occurring in the subsequence of $S$ starting at $N_1$ and ending at $N_2$ is even if $\ell=k$ and is odd if $\ell\not=k$. Therefore the tags with respect to~$\mathcal{B}'$ of all uncoloured nucleotides at a given depth share the same parity, and the parity alternates between successive tagged depths.

This completes the proof that Properties \ref{prop1} to \ref{prop4} hold for Algorithm 1.\\

We next show that Properties \ref{prop5} and \ref{prop6} hold for the initial set $\mathcal{B}$ (defined in Step 1).

Property \ref{prop5}: The initial set $\mathcal{B}$ comprises all G and C nucleotides if $n+1$ is even, or all A and U nucleotides if $n+1$ is odd, and so is balanced (by definition of the natural labelling).

Property \ref{prop6}: The labelled tree $\tau$ from which the sequence $S$ is derived shows that $S$ admits an arc set in which only nucleotides of a single type (A if $n+1$ is even, or C if $n+1$ is odd) are unpaired; therefore in the maximum-size arc set $M$ all nucleotides from the initial set $\mathcal{B}$ must pair with one another.\\

We now assume Properties \ref{prop5} and \ref{prop6} hold for the current set $\mathcal{B}$, and show that they hold for $\mathcal{B'}$.

Property \ref{prop5}: The initial application of Step 2 colours either all G and C nucleotides, or all A and U nucleotides. Since colouring is never removed, the set $\mathcal{B'}$ comprises only (a subset of the) A and U nucleotides, or G and C nucleotides, respectively. Furthermore, the set $\mathcal{B'}$ comprises all nucleotides occurring at some subset of depths not including depth $n+1$ (see Figure~\ref{steps24}(b)). Therefore, from the definition of the natural labelling, the nucleotides at this subset of depths form a new balanced~set.

Property \ref{prop6}: The elements of the balanced set $\mathcal{B}$ must all pair with one another, by Property \ref{prop6}. So each arc joining two nucleotides not in $\mathcal{B}$ must enclose an equal number of the two nucleotide types in $\mathcal{B}$ (otherwise the arc will induce a crossing). Therefore nucleotides tagged in Step 4 with the same running difference with respect to $\mathcal{B}$ can pair only with one another. In particular, nucleotides tagged in Step 4 whose running differences with respect to $\mathcal{B}$ have the same parity can pair only with one another. It follows that the elements of the set $\mathcal{B'}$ defined in Step 5 (whose tags with respect to $\mathcal{B}$ all share the same parity) can pair only with one another.

Now the labelled tree $\tau$ from which the sequence $S$ is derived shows that $S$ admits an arc set in which only nucleotides of a single type (A or C) are unpaired; therefore in the maximum-size arc set $M$ all nucleotides of the other three types must be paired. By Property \ref{prop5} we know that $\mathcal{B'}$ forms a balanced set, and we have shown that its elements can pair only with one another. Therefore the elements of $\mathcal{B'}$ must all pair with one another.

This completes the proof that Properties \ref{prop5} and \ref{prop6} hold for Algorithm 1.
\end{proof}

\begin{proof}[Proof of Claim 2.] Fix $j$ satisfying $0 \leq j \leq n$. We shall apply Algorithm 1 to $\widehat{\tau}$ with $n+1$ replaced throughout by $j$. We shall show that Properties \ref{prop1} to \ref{prop6} hold for this modified algorithm, again with $n+1$ replaced by $j$. By the same argument as previously, these properties imply that all nucleotides at depths other than $j$ must pair with one another in $M$. The nucleotides at depth $j$ can therefore pair only with one another in $M$, and the existence of $\widehat{\tau}$ (in which every nucleotide is paired) then shows that the nucleotides at depth $j$ must all pair with one another in $M$, proving Claim 2.

In order to establish that Properties \ref{prop1} to \ref{prop6} hold for the modified algorithm, we highlight only the places in which the argument differs from that given previously.

Changes throughout:
\begin{addmargin}[4em]{4em}
Replace $n+1$ by $j$, $\tau$ by $\widehat{\tau}$, and $S$ by $\widehat{S}$.
\end{addmargin}

Changes to the proof that Properties \ref{prop3} and \ref{prop4} hold for $\mathcal{B'}$:
\begin{addmargin}[4em]{4em}
Replace Figure~\ref{steps24} by Figure~\ref{rep-steps24}.
\end{addmargin}

\begin{figure}[H]
\centering
\vspace{-0.3cm}
\hspace{-1cm}
\subfloat[\hspace{1.8cm}part~(a)]{
\begin{minipage}{0.14\textwidth} 
\begin{tikzpicture}[scale=0.41,auto=left]
\node[scale=1.5](arrow) at (8,9) {$\rightarrow$};
\node[circle, scale=0.7, draw] (n01) at (1,1) {} ;
\node[left of=n01] {};
\node[right of=n01] {$1-p$};
\node[circle, scale=0.7,  draw](n11) at (1,3)  {};
\node[left of=n11] {};
\node[right of=n11] {$p$};
\node[circle, scale=0.7, draw](n21) at (1,5)  {};
\node[left of=n21] {};
\node[right of=n21] {$1-p$};
\node[circle, scale=0.7,  draw, thick](n31) at (1,7)  {$\bullet$};
\node[left of=n31] {};
\node[right of=n31] {$p$};
\node[] at (-2, 7) {Depth $j$ $\rightarrow$};
\node[circle, scale=0.7, draw](n41) at (1,9)  {};
\node[left of=n41] {};
\node[right of=n41] {$1-p$};
\node[circle, scale=0.7,  draw](n51) at (1,11)  {};
\node[left of=n51] {};
\node[right of=n51] {$p$};
\node[circle, scale=0.7,  draw](n61) at (1,13)  {};
\node[left of=n61] {};
\node[right of=n61] {$1-p$};
\node[circle, scale=0.7, draw](n71) at (1,15)  {};
\node[right of=n71] {$p$};

\foreach \from/\to in {n01/n11, n11/n21, n21/n31, n31/n41, n41/n51, n51/n61, n61/n71}
\draw [dashed] (\from) -- (\to);
\draw [dashed] (n01) -- (1, -1);
\draw [dashed] (n71) -- (1, 17);

\end{tikzpicture}
\end{minipage}
}
\hspace{3cm}
\subfloat[\hspace{0.5cm}part~(b)]{
\begin{minipage}{0.14\textwidth} 
\begin{tikzpicture}[scale=0.41,auto=left]
\node[scale=1.5](arrow) at (7,9) {$\rightarrow$};
\node[circle, scale=0.7,fill=red!30, draw] (n01) at (1,1) {} ;
\node[left of=n01] {};
\node[right of=n01] {};
\node[circle, scale=0.7,  draw](n11) at (1,3)  {};
\node[left of=n11] {};
\node[right of=n11] {};
\node[circle, scale=0.7,fill=red!30, draw](n21) at (1,5)  {};
\node[left of=n21] {};
\node[right of=n21] {};
\node[circle, scale=0.7,  draw, thick](n31) at (1,7)  {$\bullet$};
\node[left of=n31] {};
\node[right of=n31] {};
\node[circle, scale=0.7,fill=red!30, draw](n41) at (1,9)  {};
\node[left of=n41] {};
\node[right of=n41] {};
\node[circle, scale=0.7,  draw](n51) at (1,11)  {};
\node[left of=n51] {};
\node[right of=n51] {};
\node[circle, scale=0.7,fill=red!30,  draw](n61) at (1,13)  {};
\node[left of=n61] {};
\node[right of=n61] {};
\node[circle, scale=0.7, draw](n71) at (1,15)  {};
\node[] at (-2, 15) {};
\node[right of=n71] {};

\foreach \from/\to in {n01/n11, n11/n21, n21/n31, n31/n41, n41/n51, n51/n61, n61/n71}
\draw [dashed] (\from) -- (\to);
\draw [dashed] (n01) -- (1, -1);
 \draw [dashed] (n71) -- (1, 17);

\end{tikzpicture}
\end{minipage}
}
\hspace{2cm}
\subfloat[\hspace{0.5cm}part~(c)]{
\begin{minipage}{0.14\textwidth} 
\begin{tikzpicture}[scale=0.41,auto=left]
\node[circle, scale=0.7,fill=red!30, draw] (n01) at (1,1) {} ;
\node[left of=n01] {};
\node[right of=n01] {};
\node[circle, scale=0.7,  draw](n11) at (1,3)  {};
\node[left of=n11] {};
\node[right of=n11] {$1-p'$};
\node[circle, scale=0.7,fill=red!30, draw](n21) at (1,5)  {};
\node[left of=n21] {};
\node[right of=n21] {};
\node[circle, scale=0.7,  draw, thick](n31) at (1,7)  {$\bullet$};
\node[left of=n31] {};
\node[right of=n31] {$p'$};
\node[circle, scale=0.7,fill=red!30, draw](n41) at (1,9)  {};
\node[left of=n41] {};
\node[right of=n41] {};
\node[circle, scale=0.7,  draw](n51) at (1,11)  {};
\node[left of=n51] {};
\node[right of=n51] {$1-p'$};
\node[circle, scale=0.7,fill=red!30,  draw](n61) at (1,13)  {};
\node[left of=n61] {};
\node[right of=n61] {};
\node[circle, scale=0.7, draw](n71) at (1,15)  {};
\node[] at (-2, 15) {};
\node[right of=n71] {$p'$};

\foreach \from/\to in {n01/n11, n11/n21, n21/n31, n31/n41, n41/n51, n51/n61, n61/n71}
\draw [dashed] (\from) -- (\to);
\draw [dashed] (n01) -- (1, -1);
\draw [dashed] (n71) -- (1, 17);

\end{tikzpicture}
\end{minipage}
}
\caption{ }
\label{rep-steps24}
\end{figure}

Changes to the proof that Property \ref{prop6} holds for the initial set $\mathcal{B}$:
\begin{addmargin}[4em]{4em}
The labelled tree $\widehat{\tau}$ from which the sequence $\widehat{S}$ is derived shows that $\widehat{S}$ admits an arc set in which all nucleotides are paired; therefore in the maximum-size arc set $M$ all nucleotides from the initial set $\mathcal{B}$ (comprising all G and C nucleotides or all A and U nucleotides) must pair with one another.
\end{addmargin}

Changes to the proof that Property \ref{prop6} holds for $\mathcal{B'}$:
\begin{addmargin}[4em]{4em}
As previously, the elements of $\mathcal{B'}$ can pair only with one another. The labelled tree $\widehat{\tau}$ from which the sequence $\widehat{S}$ is derived shows that $\widehat{S}$ admits an arc set in which all nucleotides are paired; therefore in the maximum-size arc set $M$ all nucleotides must be paired. Therefore the elements of $\mathcal{B'}$ must all pair with one another.
\end{addmargin}
\end{proof}

Claims 1 and 2 have now been established and so the proof of Theorem \ref{main} is complete.

\section{Open Questions}
\label{sec:future}

\begin{enumerate}[Q1.]
\item 
Theorem~\ref{main} shows the existence of at least one design for the secondary structure corresponding to a P-unsaturated perfect floral tree $T$. How many such designs are there?

\item 
Theorems~\ref{old} and~\ref{tertiary} identify obstructions whose presence in a rooted tree prevents the corresponding secondary structure from being designable. Can other obstructions be identified? Can it be shown that the number of minimal obstructions is finite?

\item
Theorem~\ref{tertiary} shows that the secondary structures corresponding to rooted unsaturated trees are not designable in general, although Corollary~\ref{prune} shows that those corresponding to a subclass of rooted unsaturated trees are. Is the secondary structure corresponding to a rooted unsaturated binary tree always designable?

\item 
In Section~\ref{sec:results} we showed that Theorems~\ref{sepcol} and~\ref{main} can each be used to establish the designability of a secondary structure that the other cannot. Can these criteria be subsumed into a more general result?

\item 
Hale{\v{s}} et al.\ \cite{halevs2016combinatorial} proposed the combinatorial RNA design problem as an idealized version of the RNA design problem, in the expectation that it would produce algorithmic insights applicable to more sophisticated models (see Section~\ref{sec:intro}). As a next step, how would the results of this paper change if, instead of seeking sequences which admit a unique maximum-size arc set, we allowed sequences to admit at most two maximum-size arc sets?

\item 
As discussed in Section~\ref{sec:intro}, the practical applicability of our results is limited by the restricted class of sequences and the energy model used. To what extent can the methods of this paper be extended to investigate the designability of wider classes of sequences using more realistic energy scores?
\end{enumerate}

\section*{Acknowledgements}
Our research was carried out on the unceded Aboriginal territories of the Coast Salish people, including the Musqueam, Tsleil-Waututh, and Squamish First Nations.
Thanks to Ladislav Stacho for introducing us to this problem, to Marni Mishna for detailed comments on the manuscript, and to Stefan Hannie for interesting discussions. 
We appreciate Yann Ponty's valuable assistance in giving helpful feedback on our results and methods, providing careful responses to our questions, and suggesting Corollary~\ref{GUcorollary} and the comparison in Section~\ref{sec:results} between Theorems~\ref{sepcol} and~\ref{main}.
We are grateful to the reviewers for their very careful reading of the paper and for several constructive suggestions for improvement.


\end{document}